\documentclass{siamart171218}
\usepackage{amsfonts,amsmath,amssymb,graphicx,color}
\usepackage{psfrag}
\usepackage{pdflscape}
\usepackage{blindtext} 
\usepackage{enumitem}
\graphicspath{ {./figures/} }
\usepackage{hyperref}
\usepackage{soul}
\hypersetup{colorlinks=true}
\usepackage{algpseudocode}
\usepackage{epstopdf}
\usepackage{framed}
\usepackage{footnote}
\usepackage{pifont}
\usepackage{color}
\usepackage{todonotes}
\usepackage{amssymb}% http://ctan.org/pkg/amssymb
\usepackage{pifont}% http://ctan.org/pkg/pifont

\usepackage{yuchenmacro}

\numberwithin{figure}{section}
\numberwithin{table}{section}

\newcommand{\bequ}{\begin{equation}}     \newcommand{\eequ}{\end{equation}}
\newcommand{\benn}{\begin{equation*}}    \newcommand{\eenn}{\end{equation*}}
\newcommand{\bbma}{\begin{bmatrix}}      \newcommand{\ebma}{\end{bmatrix}}

\newcommand{\bsub}{\begin{subequations}}
	\newcommand{\esub}{\end{subequations}}

\newtheorem{thm}{Theorem}[section]

\newtheorem{lem}[thm]{Lemma}
\newtheorem{cor}[thm]{Corollary}
\newtheorem{assum}[thm]{Assumptions}
\numberwithin{equation}{section}

\renewcommand{\theequation}{\thesection.\arabic{equation}}
\newcommand{\comment}[1]{}

\newcommand{\be}{\begin{equation}}
\newcommand{\ee}{\end{equation}}

\newcommand{\bea}{\begin{eqnarray}}
\newcommand{\eea}{\end{eqnarray}}
\newcommand{\beqa}{\begin{eqnarray}}
\newcommand{\eeqa}{\end{eqnarray}}
\newcommand{\beann}{\begin{eqnarray*}}
	\newcommand{\eeann}{\end{eqnarray*}}

\newcommand{\bmat}{\left[ \begin{array}}
	\newcommand{\emat}{\end{array} \right]}

\newcommand{\beq}{\begin{equation}}
\newcommand{\eeq}{\end{equation}}

\newcommand{\bproof}{\begin{description} \item[{\it Proof}.] ~ }
	\newcommand{\eproof}{\hspace*{\fill}$\Box$\medskip \end{description}}

\newcommand{\defeq}{\stackrel{\rm def}{=}}

\newcounter{algo}[section]

		\newcounter{prog}[section]

\title{Analysis of the BFGS Method with Errors}
\author{       
        Yuchen Xie\thanks{Department of Industrial Engineering and Management Sciences, Northwestern University, 
       Evanston, IL, USA.  This author was supported by the Office of Naval Research grant N00014-14-1-0313 P00003, and by National Science Foundation grant DMS-1620022.}
       \and     
       Richard Byrd \thanks{Department of Computer Science, University of Colorado,
        Boulder, CO, USA.  This author was supported by National Science Foundation grant DMS-1620070.}
        \and
       Jorge Nocedal \thanks{Department of Industrial Engineering and Management Sciences, Northwestern University, 
       Evanston, IL, USA.  This author was supported by the Defense Advanced Research Projects Agency (DARPA). The views, opinions and/or findings expressed are those of the author and should not be interpreted as representing the official views or policies of the Department of Defense or the U.S. Government.} 
      }
\date{\today}
\begin{document}

\maketitle

%%%%%%%%%%%%%
%%% ABSTRACT %%% 
\begin{abstract}The classical convergence analysis of quasi-Newton methods assumes that the function and gradients employed at each iteration are exact. In this paper, we consider the case when there are (bounded) errors in both computations and establish conditions under which a slight modification of the BFGS algorithm with an  Armijo-Wolfe  line search converges to a neighborhood of the solution that is determined by the size of the errors. One of our results is an extension of the analysis presented in \cite{ByNoTool}, which establishes that, for strongly convex functions, a fraction of the BFGS iterates are \emph{good iterates}. We present numerical results illustrating the performance of the new BFGS method in the presence of noise.
\end{abstract}

%\newpage

%% Table of Contents
%\tableofcontents
% \newpage

\section{Introduction} \label{sec:intro}
The behavior of the BFGS method in the presence of errors has received little attention in the literature. There is, however, an increasing interest in understanding its theoretical properties and practical performance when functions and gradients are inaccurate. This interest is driven by applications where the objective function contains noise, as is the case in machine learning, and in applications where the function evaluation is a simulation subject to computational errors. 
The goal of this paper is to extend the  theory of quasi-Newton methods to the case when there are  errors in the function and gradient evaluations. We analyze the classical BFGS method with a slight modification consisting of lengthening the differencing interval as needed; all other aspects of the algorithm, including the line search, are unchanged.  We establish global convergence properties on strongly convex functions. Specifically, we show that if the errors in the function and gradient are bounded, the iterates converge to a neighborhood of the solution whose size depends on the level of noise (or error).

Our analysis builds upon the results in \cite{ByNoTool}, which identify some fundamental properties of BFGS updating. The extension to the case of inaccurate gradients is not simple due to the complex nature of the quasi-Newton iteration, where the step affects the Hessian update, and vice versa, and where the line search plays an essential role. The existing analysis relies on the observation that changes in gradients provide reliable curvature estimates, and on the fact that the line search makes decisions based on the true objective function. In the presence of errors, gradient differences can give misleading information and result in poor quasi-Newton updates. Performance can  further be impaired by the confusing effects of a line search based on inaccurate function information. We show that these  difficulties can be overcome by our modified BFGS algorithm, which performs efficiently until it reaches a neighborhood of the solution where progress is no longer possible due to errors. 

The proposed algorithm aims to be a natural adaptation of the BFGS method that is capable of dealing with noise. Other ways of achieving robustness might include update skipping and modifications of the curvature vectors, such as Powell damping \cite{Powe78}. We view these as less desirable alternatives for reasons discussed in the next section.  The line search strategy could also be performed in other ways. For example, in their analysis of a gradient method, Berahas et al.~\cite{berahas2018derivative}, relax the Armijo conditions to take noise into account. We prefer to retain the standard Armijo-Wolfe line search without any modification, as this has practical advantages.

The literature of the BFGS method with inaccurate gradients includes the implicit filtering method of Kelley et al.~\cite{choi2000superlinear, kelley2011implicit}, which assumes that noise can be diminished at will at any iteration. Deterministic convergence guarantees have been established for that method by ensuring that noise decays as the iterates approach the solution. Dennis and Walker \cite{DennWalker} and Ypma \cite{ypma} study bounded deterioration properties, and local convergence, of quasi-Newton methods with errors, when started near the solution with a Hessian approximation that is close to the exact Hessian.  
Barton \cite{barton1992computing} proposes an implementation of the BFGS method in which gradients are computed by an appropriate finite differencing technique, assuming that the noise level in the function evaluation is known. Berahas et al.~\cite{berahas2018derivative}  estimate the noise in the function using Hamming's finite difference technique \cite{hamming2012introduction}, as extended by Mor\'e and Wild \cite{more2011estimating}, and employ this estimate to compute a finite difference gradient in the BFGS method. They analyze a gradient method with a relaxation of the Armijo condition, and do not study the effects of noise in BFGS updating. 

There has recently been some interest in designing quasi-Newton methods for machine learning applications using stochastic approximations to the gradient \cite{byrd2016stochastic,gower2016stochastic,moritz2016linearly,schraudolph2007stochastic}. These papers avoid potential difficulties with BFGS or L-BFGS updating by assuming that the quality of gradient differences is always controlled, and as a result, the analysis follows similar lines as for classical BFGS and L-BFGS.

This paper is organized in 5 sections. The proposed algorithm is described in Section~\ref{sec:alg}. Section~\ref{sec:converge}, the bulk of the paper,  presents a sequence of lemmas related to the existence of  stepsizes that satisfy the Armijo-Wolfe conditions, the beneficial effect of lengthening the differencing interval, the properties of ``good iterates'', culminating in a global convergence result. Some numerical tests that illustrate the performance of the method with errors in the objective function and gradient are given in Section~\ref{sec:numerical}. The paper concludes in Section~\ref{sec:final} with some final remarks.  

\section{The Algorithm}   \label{sec:alg} 

We are interested in solving the  problem
\eqals{
	\min_{x \in \mb{R}^d} ~ \phi(x),
}
where the function  $\phi \in C^1$ and its gradient $\nabla \phi$  are not directly accessible. 
Instead, we have access to inaccurate (or noisy) versions, which we denote as $f(x)$ and $g(x)$, respectively. Thus, we write
\eqal{
        \label{fnoise}
	f(x) & = \phi(x) + \epsilon(x) \\
	g(x) & = \nabla \phi(x) + e(x),
}
where $\epsilon(x)$ and $e(x)$ define the error in function and gradient values. 
To apply the BFGS method, or a modification of it, to minimize the true function $\phi$, while observing only noisy function and gradient estimates, we must give careful consideration to the two main building blocks of the BFGS method: the line search and Hessian updating procedures.

As was shown by Powell \cite{Powe766}, an Armijo-Wolfe line search guarantees  the stability of the BFGS updating procedure, and ultimately the global convergence of the iteration (for  convex objectives). In the deterministic  case, when the smooth function $\phi(x)$ and its gradient are available,  this line search computes a stepsize $\alpha$ that satisfies:
\eqal{
	\label{eq_aw_condition}
	\phi(x + \alpha p) & \leq \phi(x) + c_1 \alpha p^T \nabla \phi(x) \qquad\mbox{ (Armijo condition)} \\
	p^T \nabla \phi(x + \alpha p)  & \geq c_2 p^T \nabla \phi(x), ~~~~~~~~~~~\qquad\mbox{(Wolfe condition)}
}
where $x$ is the current iterate, $p$ is a descent direction for $\phi$ at $x$,  (i.e., $p^T \nabla \phi(x) < 0$), and $0 < c_1 < c_2 < 1$ are user-specified parameters.  The first condition imposes sufficient decrease in the objective function, and the second requires an increase in the directional derivative (and is sometimes referred to as the \emph{curvature condition}). 
It is well known \cite{mybook} that if $\phi \in C^1$ is bounded below and has Lipschitz continuous gradients, there exists an interval of steplengths $\alpha$ that satisfy \eqref{eq_aw_condition}. 

When $\phi(x)$ and $\nabla \phi(x)$ are not accessible, it is natural to attempt to satisfy the Armijo-Wolfe conditions for the noisy function and gradient, i.e., to find $\alpha >0$ such that
\eqal{
	\label{eq_aw_condition_noised}
	f(x + \alpha p) & \leq f(x) + c_1 \alpha p^T g(x) \\
	p^T g(x + \alpha p)  & \geq c_2 p^T g(x),
}
where $p$ is the BFGS search direction. It is, however, not immediately clear whether such a stepsize  exists, and if it does,   whether it satisfies the Armijo-Wolfe conditions \eqref{eq_aw_condition} for true function $\phi$.

One possible approach to address these two challenges is to relax the Armijo-Wolfe conditions \eqref{eq_aw_condition_noised}, as is done e.g.~by Berahas et al.~\cite{berahas2018derivative} in their analysis of a gradient method with errors.  An alternative, which we adopt in this paper, is to keep the Armijo-Wolfe conditions unchanged, and show that under suitable conditions there is a stepsize that satisfies the Armijo-Wolfe  conditions for both the noisy and true objective functions. Our main assumption is that the errors $\epsilon(x), e(x)$ in \eqref{fnoise} are bounded for all $x$.
 
 Let us now consider the BFGS updating procedure. The key in the convergence analysis of quasi-Newton methods is to show that the search direction is not orthogonal to the gradient. In the literature on Newton-type methods, this is usually done by bounding the condition number of the Hessian approximation $B_k$. Whereas this  is possible for limited memory quasi-Newton methods, such as L-BFGS, in which $B_k$ is obtained by performing a  limited number of updates, one cannot bound the condition number of $B_k$ for the standard BFGS method without first proving that the iterates converge to the solution. Nevertheless, there is a result about BFGS updating  \cite{ByNoTool}, for strongly convex objective functions, whose generality will be crucial in our analysis. It states for a fixed \emph{fraction} of the BFGS iterates, the angle between the search direction and the gradient is bounded away from $90^\circ$.

To apply the results in \cite{ByNoTool}, we need to ensure that the update of $B_k$ is performed using the correction pairs 
$$[s_k, y_k]= [(x_{k+1}-x_k), (\nabla f(x_{k+1}) - \nabla f(x_k))]$$ 
that  satisfy, for all $k$, 
\eqal{
	\label{twom}
	\frac{y_k^Ts_k}{s_k^T s_k}  \geq \widehat m \cm \qquad \frac{y_k^T y_k}{y_k^T s_k} & \leq \widehat M,
}
for some constants $0 < \widehat  m \leq \widehat M$. The Armijo-Wolfe line search does not, however, guarantee that these conditions are satisfied in our setting, even under the assumption that $\phi$ is strongly convex.  To see this, note that when $\nrm{}{s_k}$ is small compared to the gradient error $\epsilon_g$, the vector $y_k$ can be contaminated by errors, and \eqref{twom} may not hold. In other words, difficulties arise when the differencing interval is too short, and to overcome this problem, we modify the ordinary BFGS method by {lengthening} the differencing interval, as needed. How to do this will be discussed in the next section.

With these ingredients in place, we provide in Algorithm~\ref{algo1} a  description of the method.  In what follows, we let $H_k$ denote the inverse Hessian approximation; i.e, $H_k = B_k^{-1}$. 
\bigskip

% \begin{algorithm}[H]              %new code 
%     %\DontqrintSemicolon
%     \caption{Outline of the BFGS Method with Errors}
%     \label{algo1}
%     \KwIn{functions $f(\cdot)$ and $ g(\cdot)$; constants $0 <c_1<c_2<1$;    lengthening parameter $l >0$; starting point $x_0 $; initial Hessian inverse approximation $H_0 \succ 0$.}
%     % \KwOut{}
%     \For{$k = 0, 1,2,...,$}{
%         $p_k \gets -H_k g(x_k)$	\\
%         Attempt to find a stepsize $\alpha^*$ such that
%     	\eqals{
%     		f(x_k + \alpha^* p_k) & \leq f(x_k) + c_1 \alpha^* p_k g(x_k) \\
%     		p_k^T g(x_k + \alpha^* p_k) & \geq c_2 p_k^T g(x_k)
%     	}\\
%         \eIf{Succeeded}{$\alpha_k \gets \alpha^*$}{$\alpha_k \gets 0$ \\}
%         \uIf{$\nrm{}{\alpha_k p_k} \geq l$}{
%             Compute the  curvature pair as usual:
%             \eqals{
%                 s_k& \gets \alpha_k p_k \\
%                 y_k& \gets g\bpa{x_k + s_k} - g(x_k)
%             }
%         }
%         \Else{
% 			Compute the curvature pair by {lengthening} the search direction:\label{algo_line_lengthening}
%             \eqals{
%                 s_k& \gets l  \frac{p_k}{\nrm{}{p_k}} \\
%                 y_k& \gets g\bpa{x_k + s_k} - g(x_k)
%             }
%         }
%         Update inverse Hessian approximation using the curvature pairs $(s_k, y_k)$: 
%         \eqal{
% 			\label{eq_bfgs_update}
%         	H_{k+1} = \bpa{I - \rho_k s_k y_k^T}H_k\bpa{I - \rho_k y_k s_k^T} + \rho_k s_k s_k^T \cm \text{where} ~ \rho_k = \frac{1}{s_k^T y_k}
%         }\\
%         $x_{k+1} \gets x_k + \alpha_k {p}_k$ \\
%     }
% \end{algorithm}

\begin{algorithm}[H]              %new code 
    %\DontqrintSemicolon
    \caption{Outline of the BFGS Method with Errors}
	\label{algo1}
	\begin{algorithmic}[1]
		\Statex \textbf{Input:} functions $f(\cdot)$ and $ g(\cdot)$; constants $0 <c_1<c_2<1$; lengthening parameter $l >0$; starting point $x_0 $; initial Hessian inverse approximation $H_0 \succ 0$.
		\For{$k = 0, 1,2,...,$}
		\State{$p_k \gets -H_k g(x_k)$}
				\State{Attempt to find a stepsize $\alpha^*$ such that
			    	\eqals{
			    		f(x_k + \alpha^* p_k) & \leq f(x_k) + c_1 \alpha^* p_k^T g(x_k) \\
			    		p_k^T g(x_k + \alpha^* p_k) & \geq c_2 p_k^T g(x_k)
					}
			}
			\If{Succeeded}
				\State{$\alpha_k \gets \alpha^*$}
			\Else
				\State{$\alpha_k \gets 0$}
			\EndIf
			\If{$\nrm{}{\alpha_k p_k} \geq l$}
				\State{Compute the  curvature pair as usual:
				            \eqals{
				                s_k \gets \alpha_k p_k,\quad y_k \gets g\bpa{x_k + s_k} - g(x_k)
							}
						}
			\Else
				\State{Compute the curvature pair by {lengthening} the search direction:\label{algo_line_lengthening}
				            \eqals{
				                s_k \gets l  \frac{p_k}{\nrm{}{p_k}},\quad y_k \gets g\bpa{x_k + s_k} - g(x_k)
							}
					}
			\EndIf
			\State{Update inverse Hessian approximation using the curvature pairs $(s_k, y_k)$: 
				\eqal{
					\label{eq_bfgs_update}
					H_{k+1} = \bpa{I - \rho_k s_k y_k^T}H_k\bpa{I - \rho_k y_k s_k^T} + \rho_k s_k s_k^T \cm \text{where} ~ \rho_k = \frac{1}{s_k^T y_k}
				}
			}
			\State{$x_{k+1} \gets x_k + \alpha_k {p}_k$}
		\EndFor
	\end{algorithmic}
\end{algorithm}

\bigskip
The only unspecified parameter in this algorithm is the lengthening parameter $l$, whose choice will be studied in the next section. We note for now that $l$ needs only be large enough to compensate for the error in the gradient, and should be at least of order $ O(\epsilon_g)$. Even though step~\ref{algo_line_lengthening} is executed when the line search fails, we will show below that the lengthening operation guarantees that $s_k^Ty_k >0$ so that the BFGS update is well defined. We also note that step~\ref{algo_line_lengthening} requires an additional gradient evaluation.

As mentioned in Section~\ref{sec:intro}, lengthening the step is not the only way to stabilize the BFGS update in the presence of errors. One alternative is to skip the update, but this can prevent the algorithm from building a useful Hessian approximation.  One can also modify the curvature vector $y_k$ when the stability of the BFGS updating cannot be guaranteed, but it is difficult to know how to design this modification in the presence of noise in the function and gradient.  We choose the lengthening approach because we view it as well suited in the presence of noise.
%, where finite difference approximations using small displacements are unreliable.

\section{Convergence Analysis}   \label{sec:converge}
In this section, we give conditions under which the BFGS method outlined above is guaranteed to yield an acceptable solution by which we mean a function value that is within the level of noise of the problem.
Throughout the paper, $\nrm{}{\cdot}$ denotes the $\ell_2$ norm.

Our analysis relies on the following assumptions regarding the  true objective function $\phi$ and the errors in function and gradients. 

\begin{assum}
\label{ass_Lip_grad_bounded_below}
The function $\phi(x)$ is bounded below and is {twice} continuously differentiable with an $M$-Lipschitz continuous {($M > 0$)} gradient, i.e.,
\eqals{
	\nrm{}{\nabla \phi(x) - \nabla \phi(y)} \leq M \nrm{}{x-y} \cm \forall x, y \in \mb{R}^d .
}
\end{assum}
\noindent
This assumption could be relaxed to require only that the gradients be Lipschitz continuous; we make the stronger assumption that $\phi \in {C}^2$ only to simplify the proof of one of the lemmas below.

\begin{assum}
\label{ass_bounded_noise}
The errors in function and gradients values are uniformly bounded, i.e., $\forall x \in \mb{R}^d$, there exist {non-negative} constants $\epsilon_f, \epsilon_g$ such that
\eqals{   
	| f(x)-  \phi(x)|= \abs{\epsilon(x)} & \leq \epsilon_f \\
	\| g(x) - \nabla \phi(x) \| = \nrm{}{e(x)} & \leq \epsilon_g.
}
\end{assum}

There are many applications where this assumption holds; one of the most prominent is the case of computational noise that arises when the evaluation of the objective function involves an adaptive numerical computation \cite{more2011estimating}. On the other hand, there are other applications where Assumption~\ref{ass_bounded_noise} is not satisfied, as is the case when errors are due to Gaussian noise. Nevertheless, since the analysis for  unbounded errors appears to be complex \cite{paquette}, we will not consider it  here, as our main goal is to advance our understanding of the BFGS method in the presence of errors, and this is best done, at first, in a  benign setting.

\subsection{Existence of  Armijo-Wolfe Stepsizes}
We  begin our analysis by presenting a result that will help us establish the existence of stepsizes satisfying the Armijo-Wolfe conditions.  Since we will impose these conditions on  the noisy functions (i.e. \eqref{eq_aw_condition_noised}) and want to show that they also apply to the true function, the following lemma  considers two sets of functions and gradients:  $F_A$ and $G_A$ can be viewed as proxies for the true function and gradient $\phi$ and $\nabla \phi$, while $F_B$ and $G_B$ stand for the approximate function $f$ and its gradient approximation {$g$}. (In a later lemma these roles are reversed.)
It is intuitively clear, that the Armijo-Wolfe conditions can only be meaningful when the gradients are not dominated by  errors. Therefore, our first lemma shows that when the gradients $G_A, G_B$ are sufficiently {large compared to}  $\epsilon_f, \epsilon_g$,  the Armijo-Wolfe conditions can be satisfied.

Below, we let $\varphi$ denote the angle between a vector $p \in \mb{R}^d$ and a vector $-G \in \mb{R}^d$, i.e.,
\eqal{ \label{cosdef}
  \varphi = \angle (p, -G) \quad\mbox{or} \quad \cos \varphi = \frac{-p^T G}{\|p\|\|G\|}.
   }
   In the sequel, $\varphi_A, \varphi_B$ denote the angles obtained by substituting $G_A, G_B$ in this definition.

\begin{lem}
\label{lem_aw_with_noise}
Suppose that a scalar function $F_A: \mb{R}^d \to \mb{R}$ is continuous and bounded below, and that a vector function $G_A: \mb{R}^d \to \mb{R}^d$ satisfies 
\eqal{
	\label{eq_thm_ass_quasi_lip}
	\nrm{}{G_A(y) - G_A(z)} \leq L \nrm{}{y-z} + \Lambda \cm \quad \forall y, z \in \mb{R}^d \cm
}
for some constants {$L > 0, \Lambda \geq 0$}.
Suppose $x \in \mb{R}^d$ is such that $G_A(x) \neq 0$,  that $p \in \mb{R}^d$ satisfies $p^T G_A(x) < 0$, and  that the stepsize $\alpha > 0$ satisfies the Armijo-Wolfe conditions
\eqal{
	\label{eq_thm_ass_aw_cond}
	F_A(x + \alpha p) & \leq F_A(x) + c_{A1} \alpha p^T G_A(x) \\
	p^T G_A(x + \alpha p) & \geq c_{A2} p^T G_A(x),
}
for $0< c_{A1}<c_{A2}<1$. 
Furthermore, suppose another scalar function $F_B: \mb{R}^d \to \mb{R}$ and vector function $G_B: \mb{R}^d \to \mb{R}^d$ satisfy
\eqal{
	\label{eq_thm_ass_diff}
	\abs{F_A(y) - F_B(y)} \leq & \epsilon_f \cm \forall y \in \mb{R}^d\\
	\nrm{}{G_A(y) - G_B(y)} \leq & \epsilon_g \cm \forall y\in \mb{R}^d ,
}
for some non-negative constants $\epsilon_f, \epsilon_g$.
Assume that $G_B(x) \neq 0$ and that $p$ satisfies $p^T G_B(x) < 0$. 
%Let $\varphi_A$ be the acute angle between $-G_A(w)$ and $q$; let $\varphi_B$ be the acute angle between $-G_B(w)$ and $q$. Let $\tau_{A,B}(w) = \nrm{}{G_A(w)}/\nrm{}{G_B(w)}$. 
Let $\gamma_1, \gamma_2$ be two constants such that
\begin{equation} \label{gammas}
    0 < \gamma_1 < c_{A1} \quad\mbox{and} \quad 0 < \gamma_2 < 1 - c_{A2}.
\end{equation} 
If the following conditions hold:
\eqal{ 
	\label{3cond}
	\nrm{}{G_A(x)} & \geq {\frac{2\Lambda}{(1-c_{A2}) \cos \varphi_A}} \\
	\nrm{}{G_B(x)} & \geq \max \bst{\frac{2 c_{A1} \epsilon_g}{\gamma_1 \cos \varphi_B }\cm \frac{(1+c_{A2})\epsilon_g}{\gamma_2 \cos \varphi_B}} \\
	\nrm{}{G_A(x)}  \nrm{}{G_B(x)} & \geq {\frac{8 L \epsilon_f}{\gamma_1 (1-c_{A2}) \cos \varphi_A \cos \varphi_B}} ,
}
then the stepsize $\alpha$ satisfies the Armijo-Wolfe conditions with respect to $F_B$ and $G_B$:
\begin{align}
	F_B(x + \alpha p) & \leq F_B(x) + (c_{A1} - \gamma_1) \alpha p^T G_B(x)   \label{tray1} \\
	p^T G_B(x + \alpha p) & \geq (c_{A2} + \gamma_2) p^T G_B(x).   \label{tray2}
\end{align}
\end{lem}

\begin{proof}
By the second equation in \eqref{eq_thm_ass_aw_cond}, i.e.,
\eqals{
	p^T G_A(x + \alpha p) & \geq c_{A2} p^T G_A(x) \cm
}
we have
\eqals{
	-(1-c_{A2}) p^T G_A(x) \leq p^T \big(G_A(x+ \alpha p \big) - G_A(x)) .
}
Using \eqref{eq_thm_ass_quasi_lip} we have
%\eqals{
%	\nrm{}{G_A(w+ \alpha q) - G_A(w)} \leq \alpha L \nrm{}{q} + \Lambda .
%}
%and thus
\eqals{
	-(1-c_{A2}) p^T G_A(x) \leq 
%	q^T (G_A(w+ \alpha q) - G_A(w)) \leq 
	\nrm{}{p} \bpa{ \alpha L \nrm{}{p} + \Lambda}.
}
Recalling the definition \eqref{cosdef}, we obtain the lower bound
\eqals{
	\alpha \geq  \frac{(1-c_{A2})\cos \varphi_A \nrm{}{G_A(x)} - \Lambda}{L \nrm{}{p}}.
}
From \eqref{3cond} we have
% \eqals{
% 	\nrm{}{G_B(w)} \geq \frac{2\Lambda}{(1-c_{A2})\tau_{A, B}(w) \cos \varphi_A},
% }
\eqals{
	\nrm{}{G_A(x)} \geq \frac{2\Lambda}{(1-c_{A2}) \cos \varphi_A},
}
i.e.,
\eqals{
	(1-c_{A2}) \cos \varphi_A \nrm{}{G_A(x)} \geq 2 \Lambda ,
}
% and recalling the definition \eqref{tau} this inequality can be written as
% \eqals{
% 	(1-c_{A2}) \cos \varphi_A \nrm{}{G_A(w)} \geq 2 \Lambda ,
% }
from which it follows that
\eqals{
	\alpha \geq \underline{\alpha} \defeq \frac{(1-c_{A2})\cos \varphi_A \nrm{}{G_A(x)}}{2L \nrm{}{p}} .
}
Now, by \eqref{3cond} we also have
% \eqals{
% 	\nrm{}{G_B(w)} \geq ~ & \sqrt{\frac{8 L \epsilon_f}{\tau_{A, B}(w) \gamma_1 (1-c_{A2}) \cos \varphi_A \cos \varphi_B}} 
% }
\eqals{
	\nrm{}{G_A(x)}  \nrm{}{G_B(x)} & \geq {\frac{8 L \epsilon_f}{\gamma_1 (1-c_{A2}) \cos \varphi_A \cos \varphi_B}} ,
}
and thus
\begin{eqnarray}
	 -\gamma_1 {\alpha} p^T G_B(x) 
	& \geq & -\gamma_1 \underline{\alpha} p^T G_B(x) \nonumber \\
	& =  & \gamma_1 \frac{(1-c_{A2})\cos \varphi_A \nrm{}{G_A(x)}}{2L \nrm{}{p}} \nrm{}{p} \nrm{}{G_B(x)} \cos \varphi_B  \nonumber \\
	 & = & \frac{\gamma_1 (1-c_{A2})\cos \varphi_A \cos \varphi_B}{2L } \nrm{}{G_A(x)}\nrm{}{G_B(x)} 
	\geq  4 \epsilon_f . \label{eq_prf_ac_bound_1}
\end{eqnarray}
From \eqref{3cond} 
\eqals{
	\nrm{}{G_B(x)} \geq ~ &\frac{2 c_{A1} \epsilon_g}{\gamma_1 \cos \varphi_B} \cm
}
or 
\eqal{
	\label{eq_prf_ac_bound_2}
	-\gamma_1 {\alpha} p^T G_B(x) \geq 2 c_{A1} \alpha \nrm{}{p} \epsilon_g .
}
Adding \eqref{eq_prf_ac_bound_1} and \eqref{eq_prf_ac_bound_2} yields
\eqal{
	\label{eq_prf_ac_bound_comb}
	-\gamma_1 {\alpha} p^T G_B(x) \geq 2 \epsilon_f + c_{A1} \alpha \nrm{}{p} \epsilon_g.
}
The first inequality in \eqref{eq_thm_ass_aw_cond} and Assumptions~\eqref{eq_thm_ass_diff} give
\eqals{
	F_B(x + \alpha p) & \leq F_B(x) + c_{A1} \alpha p^T G_B(x) + 2 \epsilon_f + c_{A1} \alpha \nrm{}{p} \epsilon_g  \cm
}
which combined with \eqref{eq_prf_ac_bound_comb} yields
\eqal{
	F_B(x + \alpha p) 
%	& \leq F_B(w) + c_{A1} \alpha q^T G_B(w) + 2 \epsilon_f + c_{A1} \alpha \nrm{}{q} \epsilon_g
	& \leq F_B(x) + (c_{A1} - \gamma_1) \alpha p^T G_B(x) .
}
This proves \eqref{tray1}. 

Next, by \eqref{3cond}
\eqals{
	\nrm{}{G_B(x)} \geq \frac{(1+c_{A2})\epsilon_g}{\gamma_2 \cos \varphi_B} \cm
}
or
\eqal{
	\label{eq_prf_wc_bound_comb}
	-(1+c_{A2}) \epsilon_g \nrm{}{p} \geq \gamma_2 p^T G_B(x) .
}
By the second equation in \eqref{eq_thm_ass_aw_cond} and assumption \eqref{eq_thm_ass_diff} we immediately have
\eqals{
	p^T G_B(x + \alpha p) & \geq c_{A2} p^T G_B(x) - (1+c_{A2}) \epsilon_g \nrm{}{p} .
}
Then by \eqref{eq_prf_wc_bound_comb} we have
\eqals{
	p^T G_B(x + \alpha p) 
	%& \geq c_{A2} q^T G_B(w) - (1+c_{A2}) \epsilon_g \nrm{}{q} 
	& \geq (c_{A2} + \gamma_2) p^T G_B(x) ,
}
which proves  \eqref{tray2}.
\end{proof}
% Finally, it is easy to see that the bounds on $\nrm{}{G_A(w)}$ and  $\nrm{}{G_B(w)}$ are equivalent, by definition \eqref{tau} of $\tau_{A,B}(w)$. 

Note that there is some flexibility in the choice of $\gamma_1, \gamma_2$ in \eqref{gammas}, which influences the constants in \eqref{3cond}. 
This lemma gives conditions under which the Armijo-Wolfe conditions hold, but the  bounds \eqref{3cond}, 
involve the angles $\varphi_A, \varphi_B$, which have not been shown to be bounded away from $90^\circ$ (so that the cosine terms are not bounded away from zero). Hence, this result is preliminary. We continue the analysis leaving the angles $\varphi_A, \varphi_B$ as parameters to be bounded later. 

 In the sequel, we let $g_k = g(x_k)$, define $\theta_k$ to be the angle between $p_k$ and $-g_k$, and $\wt{\theta}_k$  the angle between $p_k$ and $-\nabla \phi(x_k)$, i.e.,
\begin{align}  
\theta_k = \angle (-p_k, g_k) \quad\mbox{or} \quad  \cos(\theta_k)= &  -p_k^Tg_k/\|p_k\| \|g_k\|   \label{twotheta}\\
\wt{\theta}_k = \angle (-p_k, \phi(x_k)) \quad\mbox{or} \quad \cos(\wt{\theta}_k)= &  -p_k^T \nabla \phi(x_k)/\|p_k\| \|\nabla \phi(x_k)\| . \label{twotheta2}
\end{align}
We now use Lemma~\ref{lem_aw_with_noise} to establish the existence of Armijo-Wolfe stepsizes for the noisy function and gradient, $f$ and $g$, under the assumption that the  true gradient $\nabla \phi$ is not too small.

\begin{thm}
\label{thm_existence_of_aw_step}
Suppose that Assumptions \ref{ass_Lip_grad_bounded_below} and \ref{ass_bounded_noise} hold, and that at iteration $k$ the search direction $p_k$ satisfies $p_k^T g_k < 0$. 
Let 
$0 < c_1 < c_2 < 1$ and $0 < \delta_1 < 1$, $0 < \delta_2 < 1$ be constants such that  $\delta_1 + \delta_2 < c_2 - c_1$. 
If
\eqal{ \label{who}
	\nrm{}{\nabla \phi(x_k)}  \geq \max \Bigg \{ & \frac{4 ({c}_1 + \delta_1)\epsilon_g}{\delta_1 \cos \theta_k} \cm \frac{2(1+{c}_2-\delta_2)\epsilon_g}{\delta_2 \cos \theta_k} \cm \\
	& \sqrt{{\frac{16M \epsilon_f}{(1-{c}_2+\delta_2) \delta_1 \cos \theta_k \cos \wt{\theta}_k}}}
	\Bigg\}  ,
}
 there exists a stepsize $\alpha_k$ such that 
\eqal{
	\label{eq_thm_existence_of_aw_step_aw_cond}
	f(x_k + \alpha_k p_k) & \leq f(x_k) + c_1 \alpha_k p_k^T g(x_k) \\
	p_k^T g(x_k + \alpha_k p_k) & \geq c_2 p_k^T g(x_k) .
}
\end{thm}

\begin{proof}
We invoke Lemma \ref{lem_aw_with_noise} with  $x \gets x_k$, $F_A(\cdot) \gets \phi(\cdot)$, $G_A(\cdot) \gets \nabla \phi(\cdot)$, $F_B(\cdot) \gets f(\cdot)$, $G_B(\cdot) \gets g(\cdot)$, and $p \gets p_k$. Then, from \eqref{twotheta}-\eqref{twotheta2} we have that $\varphi_A = \wt{\theta}_k$ and $\varphi_B = {\theta}_k$. Let 
% \yx{\st{$c_{B1} = c_1$, $c_{B2} = c_2$, let }}
$\gamma_1 = \delta_1, \gamma_2 = \delta_2$; $c_{A1} = c_1 + \delta_1$ and $c_{A2} = c_2 - \delta_2$. Our assumptions on $\delta_1, \delta_2, c_1, c_2$ imply that $0 < c_{A1} < c_{A2} < 1$, and that conditions \eqref{gammas} hold. 
 
We must verify that the assumptions of Lemma \ref{lem_aw_with_noise} are satisfied.
By Assumption~\ref{ass_Lip_grad_bounded_below}, $F_A$ is bounded below and
\eqals{
	\nrm{}{G_A(y) - G_A(z)} \leq M \nrm{}{y-z} ,
}
so that \eqref{eq_thm_ass_quasi_lip} holds with $L = M$ and $\Lambda = 0$. We assume that $p^T G_B(x) = p_k^T g_k < 0$. To show that $p^T G_A(x) < 0$, note that by \eqref{who} 
\eqals{
	\nrm{}{\nabla \phi(x_k)} \geq \frac{4 ({c}_1 + \delta_1)}{\delta_1 } \frac{\epsilon_g}{\cos \theta_k} > 2\epsilon_g .
}
By Assumption~\ref{ass_bounded_noise}, we have that $ \nrm{}{\nabla \phi(x_k)-g_k} \leq \epsilon_g$. Therefore,
\eqal{ \label{makes}
	\nrm{}{g(x_k)} \geq \nrm{}{\nabla \phi(x_k)} - \epsilon_g \geq \frac{1}{2} \nrm{}{\nabla \phi(x_k)} .
}
We also have that
\eqals{
	\nrm{}{g(x_k)} \geq \frac{1}{2} \nrm{}{\nabla \phi(x_k)} 
	\geq \frac{2(c_1 + \delta_1)\epsilon_g}{\delta_1 \cos \theta_k} 
	> \frac{\epsilon_g}{\cos \theta_k} ,
}
or 
\eqals{
	\nrm{}{g(x_k)} \cos \theta_k > \epsilon_g .
}
Recalling again Assumption~\ref{ass_bounded_noise}, this bound yields
\eqals{
	p^T G_A(x) & \leq p^T G_B(x) + \nrm{}{p} \epsilon_g \\
	%& = \nrm{}{p} \nrm{}{G_B(x)} \cos \phi_B - \nrm{}{p} \epsilon_g \\
	& = -\nrm{}{p} \bpa{\nrm{}{G_B(x)} \cos \varphi_B - \epsilon_g} \\
	& = -\nrm{}{p_k} \bpa{\nrm{}{g_k} \cos \theta_k - \epsilon_g} \\
	& < 0 .
}

Knowing that $p_k$ is a descent direction for the true function $\phi$, and since $\phi$ is continuously differentiable and bounded from below, we can guarantee \cite{mybook} the existence of a stepsize $\alpha = \alpha_k$ such that 
\eqals{
	F_A(x+\alpha p) & \leq F_A(x) + c_{A1} \alpha p^T G_A(x) \\
	p^T G_A(x+\alpha p) & \geq c_{A2} p^T G_A(x),
}
showing that \eqref{eq_thm_ass_aw_cond} is satisfied. 

To prove that \eqref{eq_thm_existence_of_aw_step_aw_cond} holds, all that is necessary is to show that \eqref{who} implies conditions \eqref{3cond}. The first condition is immediately satisfied, since we have shown that we can choose $\Lambda = 0$.
By the definitions given in the first paragraph of this proof, the other two conditions in \eqref{3cond} can be written as  
\eqal{
	\label{eq_thm_existence_of_aw_step_required_cond_raw}
	\nrm{}{g(x_k)} & \geq \max \bst{\frac{2 ({c}_1 + \delta_1)}{\delta_1 } \cm \frac{(1+{c}_2 {- \delta_2)}}{\delta_2 }}  \frac{\epsilon_g}{\cos \theta_k} \\
	\nrm{}{\nabla \phi(x_k)}  \nrm{}{g(x_k)} & \geq {{\frac{8M \epsilon_f}{(1-{c}_2+\delta_2) \delta_1 \cos \theta_k \cos \wt{\theta}_k}}} .
}
% Now we use the following lemma
% \begin{lem}
% \label{lem_bound_g_using_nabla_phi}
% Suppose assumption \ref{ass_bounded_noise} are satisfied and we have 
% \eqals{
% 	\nrm{}{\nabla \phi(x)} \geq 2 \epsilon_g
% }
% Then
% \eqals{
% 	\frac{1}{2} \nrm{}{\nabla \phi(x)} \leq \nrm{}{g(x_k)} \leq \frac{3}{2} \nrm{}{\nabla \phi(x)}
% }
% \end{lem}
% \begin{proof}[Proof of Lemma \ref{lem_bound_g_using_nabla_phi}]
% We have
% \eqals{
% 	\frac{\nrm{}{g(x)}}{\nrm{}{\nabla \phi(x)}} \leq \frac{\nrm{}{\nabla \phi(x)} + \epsilon_g}{\nrm{}{\nabla \phi(x)}} \leq \frac{3}{2}
% }
% and
% \eqals{
% 	\frac{\nrm{}{g(x)}}{\nrm{}{\nabla \phi(x)}} \geq \frac{\nrm{}{\nabla \phi(x)} - \epsilon_g}{\nrm{}{\nabla \phi(x)}} \geq \frac{1}{2}
% }
% \end{proof}
To see that these two conditions hold, we first note that by \eqref{makes},
\eqals{
	\nrm{}{g(x_k)} \geq \frac{1}{2} \nrm{}{\nabla \phi(x_k)} \geq \max \bst{\frac{2 ({c}_1 + \delta_1)}{\delta_1 } \frac{\epsilon_g}{\cos \theta_k} \cm \frac{(1+{c}_2-\delta_2)}{\delta_2 } \frac{\epsilon_g}{\cos \theta_k}}.
}
Also, from \eqref{who}
\eqals{
	\nrm{}{\nabla \phi(x_k)}  \nrm{}{g(x_k)} \geq \frac{1}{2} \nrm{}{\nabla \phi(x_k)}^2 \geq {\frac{8M \epsilon_f}{(1-{c}_2+\delta_2) \delta_1 \cos \theta_k \cos \wt{\theta}_k}} \, .
}
Hence, all the conditions of Lemma~\ref{lem_aw_with_noise} are satisfied, and we conclude that there exists a stepsize $\alpha$ that satisfies \eqref{eq_thm_existence_of_aw_step_aw_cond}.
\end{proof}

% It is sometimes desirable to express the bounds on $\nrm{}{\nabla \phi(x_k)}$ only. To this end, observe that when $\nrm{}{\nabla \phi(x_k)}$ is sufficiently large w.r.t. $\epsilon_g$, we can lower- and upper-bound $\nrm{}{g(x_k)}$ by a constant times $\nrm{}{\nabla \phi(x_k)}$, based on the fact that $\nrm{}{\nabla \phi(x) - g(x)} \leq \epsilon_g$. We have the following very simple argument:

In the previous theorem we gave conditions under which the Armijo-Wolfe conditions are satisfied with respect to $f$ and $g$. We now use Lemma~\ref{lem_aw_with_noise} to show that satisfaction of the Armijo-Wolfe conditions for the approximate function $f$ implies satisfaction for the true objective $\phi$, under certain conditions.

\begin{thm}
\label{thm_goodness_of_aw_step} Suppose Assumptions \ref{ass_Lip_grad_bounded_below} and \ref{ass_bounded_noise} are satisfied, and that at iteration $k$ the search direction $p_k$ satisfies $p_k^T g_k < 0$. 
Let $\theta_k$ and $\wt{\theta}_k$ be defined by \eqref{twotheta}, \eqref{twotheta2}.
%Let $\theta_k$ be the angle between $p_k$ and $-g_k$, and $\wt{\theta}_k$ be the angle between $p_k$ and $-\nabla \phi(x_k)$. 
Let $0 < c_1 < c_2 < 1$, {and $\wh \delta_1, \wh \delta_2$} be constants such that $0 < \wh{\delta}_1 < c_1$, $0 < \wh{\delta}_2 < 1 - c_2$.
Suppose there exists a stepsize $\alpha_k$ such that
\eqals{
	f(x_k + \alpha_k p_k) & \leq f(x_k) + c_1 \alpha_k p_k^T g(x_k) \\
	p_k^T g(x_k + \alpha_k p_k) & \geq c_2 p_k^T g(x_k).
}
If
\eqal{  \label{tennis}
	\nrm{}{\nabla \phi(x_k)} \geq \max \Bigg\{ &\frac{8 \epsilon_g}{(1-c_2) \cos \theta_k} \cm \sqrt{\frac{16 M \epsilon_f}{\wh{\delta}_1 (1-c_2) \cos \theta_k \cos \wt{\theta}_k}} \cm \\
	& \frac{2 c_1 \epsilon_g}{\wh{\delta}_1 \cos \wt{\theta}_k} \cm \frac{(1+c_2) \epsilon_g}{\wh{\delta}_2 \cos \wt{\theta}_k}\Bigg\},
}
% \eqals{
% 	\nrm{}{\nabla \phi(x_k)} \geq \max \bst{\frac{4 \tau_k \epsilon_g}{(1-c_2) \cos \theta_k} \cm \sqrt{\frac{8 M \tau_k \epsilon_f}{\wh{\delta}_1 (1-c_2) \cos \theta_k \cos \wt{\theta}_k}} \cm \frac{2 c_1 \epsilon_g}{\wh{\delta}_1 \cos \wt{\theta}_k} \cm \frac{(1+c_2) \epsilon_g}{\wh{\delta}_2 \cos \wt{\theta}_k}}.
% }
then $\alpha_k$  satisfies
\eqal{
	\label{eq_thm_goodness_of_aw_step_aw_condition}
	\phi(x_k + \alpha_k p_k) & \leq \phi(x_k) + (c_1 - \wh{\delta}_1) \alpha_k p_k^T \nabla \phi(x_k) \\
	p_k^T \nabla \phi(x_k + \alpha_k p_k) & \geq (c_2 + \wh{\delta}_2) p_k^T \nabla \phi(x_k) .
}
\end{thm}

\begin{proof}
We prove this by applying Lemma \ref{lem_aw_with_noise}, reversing the roles of $F_A, F_B$, compared to Lemma~\ref{thm_existence_of_aw_step}. 
Specifically, we now let $x \gets x_k$, $F_A(\cdot) \gets f(\cdot)$, $G_A(\cdot) \gets g(\cdot)$, $F_B(\cdot) \gets \phi(\cdot)$, $G_B(\cdot) \gets \nabla \phi(\cdot)$, and $p \gets p_k$. We define $\varphi_A = {\theta}_k$ and $\varphi_B = \wt{\theta}_k$ as in \eqref{twotheta}, \eqref{twotheta2}.   Let $c_{A1} = c_1$, $c_{A2} = c_2$; $\gamma_1 = \wh{\delta}_1, \gamma_2 = \wh{\delta}_2$. Clearly we have $0 < c_{A1} <$ $c_{A2} < 1$.% \yx{\st{; then $c_{B1} = c_1 - \wh{\delta}_1$ and $c_{B2} = c_2 + \wh{\delta}_2$}}. 

We need to verify that the assumptions of Lemma~\ref{lem_aw_with_noise} are satisfied. By Assumptions~\ref{ass_Lip_grad_bounded_below} and \ref{ass_bounded_noise}  we have
\eqals{
	\nrm{}{G_A(y) - G_A(z)} = \nrm{}{g(y) - g(z)} \leq \nrm{}{\nabla \phi(y) - \nabla \phi(z)} + 2\epsilon_g \leq M \nrm{}{y-z} + 2 \epsilon_g ,
}
and hence  Assumption~\eqref{eq_thm_ass_quasi_lip} is satisfied with $L = M$ and $\Lambda = 2 \epsilon_g$. 

We assume that  $p^T G_A(x) = p^Tg_k< 0$. To show that $p^T G_B(x) < 0$, we note from \eqref{tennis} that
\eqals{
	\nrm{}{\nabla \phi(x_k)} \geq \frac{8 \epsilon_g}{(1-c_2) \cos \theta_k} > 2\epsilon_g ,
}
and as in \eqref{makes}
\eqals{
	\nrm{}{g(x_k)} \geq \nrm{}{\nabla \phi(x_k)} - \epsilon_g \geq \frac{1}{2} \nrm{}{\nabla \phi(x_k)} .
}
Therefore,
\eqal{  \label{court}
	\nrm{}{g(x_k)} \geq \frac{1}{2} \nrm{}{\nabla \phi(x_k)} \geq \frac{4 \epsilon_g}{(1-c_2) \cos \theta_k} > \frac{\epsilon_g}{\cos \theta_k},
}
i.e,
\eqals{
	\nrm{}{g_k} \cos \theta_k > \epsilon_g.
}
Now,
\eqals{
	p^T G_B(x) & \leq p^T G_A(x) + \nrm{}{p} \epsilon_g \\
	& = - \nrm{}{p} \bpa{\nrm{}{G_A(x)} \cos \varphi_A - \epsilon_g} \\
	& = - \nrm{}{p_k} \bpa{\nrm{}{g_k} \cos \theta_k - \epsilon_g} \\
	& < 0 .
}
It remains to show that conditions \eqref{3cond} are satisfied, from which it would follow that  $\alpha_k$ satisfies \eqref{eq_thm_goodness_of_aw_step_aw_condition}, proving the theorem. Since  $\Lambda = 2 \epsilon_g$, conditions \eqref{3cond} read, in the notation of this lemma,

\eqal{
	\label{eq_thm_goodness_of_aw_step_requirement}
	\nrm{}{g(x_k)} & \geq \frac{4 \epsilon_g}{(1-c_2) \cos \theta_k} \\
	\nrm{}{\nabla \phi(x_k)} & \geq \max \bst{\frac{2 c_1}{\wh{\delta}_1} \cm \frac{(1+c_2)}{\wh{\delta}_2}}  \frac{\epsilon_g}{\cos \wt{\theta}_k} \\
	% \nrm{}{\nabla \phi(x_k)} & \geq \max \bst{\frac{2 c_1 \epsilon_g}{\wh{\delta}_1 \cos \wt{\theta}_k} \cm \frac{(1+c_2) \epsilon_g}{\wh{\delta}_2 \cos \wt{\theta}_k}} \\
	\nrm{}{\nabla \phi(x_k)}  \nrm{}{g(x_k)} & \geq {\frac{8 M \epsilon_f}{\wh{\delta}_1 (1-c_2) \cos \theta_k \cos \wt{\theta}_k}}.
}
We have already shown,  in \eqref{court}, the first condition, and the second condition follows from Assumption\eqref{tennis}. Finally, from \eqref{court} and \eqref{tennis},
%\eqals{
%	\nrm{}{\nabla \phi(x_k)} \geq \frac{8 \epsilon_g}{(1-c_2) \cos \theta_k} > 2\epsilon_g
%}
%we again have
%\eqals{
%	\nrm{}{g(x_k)} \geq \nrm{}{\nabla \phi(x_k)} - \epsilon_g \geq \frac{1}{2} \nrm{}{\nabla \phi(x_k)}
%}
%Then, one can check that
%\eqals{
%	\nrm{}{g(x_k)} \geq \frac{1}{2} \nrm{}{\nabla \phi(x_k)} \geq \frac{4 \epsilon_g}{(1-c_2) \cos \theta_k}
%}
%Also, 
\[
	\nrm{}{\nabla \phi(x_k)} \nrm{}{g(x_k)} \geq \frac{1}{2} \nrm{}{\nabla \phi(x_k)}^2 \geq  {\frac{8 M \epsilon_f}{\wh{\delta}_1 (1-c_2) \cos \theta_k \cos \wt{\theta}_k}}.
\]
\end{proof}

Theorems \ref{thm_existence_of_aw_step} and \ref{thm_goodness_of_aw_step} establish the existence of a neighborhood of the solution, defined in terms of $\nrm{}{\nabla \phi(x)}$, outside of which the Armijo-Wolfe line search strategy is well defined. 
%This is possible because the error in function  and gradients values is assumed to be bounded by $\epsilon_f$ and $\epsilon_g$, and hence when when $\nrm{}{\nabla \phi(x)}$ is sufficiently large with respect to these errors the the Armijo-Wolfe line search is unaffected. 
This neighborhood depends on  $\epsilon_f$ and $\epsilon_g$, as well as $\cos \theta_k$ and $\cos \wt{\theta}_k$ --- and the latter two quantities have not yet been bounded away from zero. Thus, similar to the central role that $\cos \wt{\theta}_k$ plays in the classic convergence analysis of gradient methods, $\cos \theta_k$ and $\cos \wt{\theta}_k$ play a key role in the convergence analysis of our algorithm presented below.

\subsection{Lengthening the Differencing Interval}  \label{longer}
The BFGS method is complex in that Hessian updates affect the search direction and vice versa. As a result, it is not possible to show that the condition number of the Hessian approximations $B_k$ is bounded, without first showing convergence of the iterates.
%As can be seen in theorem \ref{thm_existence_of_aw_step} and \ref{thm_goodness_of_aw_step}, the key of the convergence analysis is to obtain a lower bound of $\cos \theta_k$. In literatures on quasi-Newton methods, this is usually done by bounding the condition number of Hessian inverse approximation $H_k$. Unfortunately, this approach is only possible for methods where $H_k$ is obtained by performing a \st{uniformly} bounded number of updates, such as L-BFGS. 
%For BFGS updates, it can be shown that if the iterates does not converge, then there are cases where the condition number of $H_k$ grows without bound. 
Nevertheless,  it is has been shown \cite{ByNoTool}  that under mild assumptions,  the angle between the search direction and the negative gradient can be bounded away from zero for a fraction of the iterates, which is  sufficient to establish R-linear convergence. 

To apply the results in \cite{ByNoTool}, the curvature pairs $(s_k, y_k)$ used to update $H_k$ must satisfy  
\eqal{
	\label{eq_tool_paper_condition}
	\frac{y_k^Ts_k}{s_k^T s_k}  \geq \widehat m \cm \qquad \frac{y_k^T y_k}{y_k^T s_k} & \leq \widehat M , \qquad \forall k \cm
}
for some constants $0 < \widehat m \leq \widehat M$. These conditions will not generally hold unless we make the following additional assumption. 
\begin{assum}  \label{ass_strong_convex}
     The function $\phi$ is $m$-strongly convex, with $0 < m \leq M$.
     (Recall that $M$ is defined in Assumptions~\ref{ass_Lip_grad_bounded_below}.)
\end{assum}
Assumptions~\ref{ass_Lip_grad_bounded_below}, \ref{ass_bounded_noise} and \ref{ass_strong_convex} are still not sufficient to establish \eqref{eq_tool_paper_condition} because, if {$\nrm{}{{s}_k}$} is small  compared to the error in the gradient, $\epsilon_g$,  then {the vector $y_k$ can be highly unreliable.} To overcome this, we increase the differencing interval and recompute the gradient before performing the BFGS update, as stipulated in Algorithm~\ref{algo1}, i.e., we set
\[
s_k \gets l  \frac{p_k}{\nrm{}{p_k}}, \qquad
                y_k \gets g\bpa{x_k + s_k} - g(x_k), \quad l>0.
 \]
 We show below that if $l$ is sufficiently large, these conditions ensure that \eqref{eq_tool_paper_condition} holds. 
  Lemma~\ref{lemma_effect_of_lengthening} identifies the minimum value of  $l$. Before presenting that result, we need the following technical lemma, whose proof is given in the Appendix. In what follows, $\lambda(H) $ denotes the set of eigenvalues of a matrix $H$.

\begin{lem}
\label{lemma_mu_L_smooth}
Let $s, y \in \mb{R}^d$ be two non-zero vectors, and let $0 < \mu \leq L$. 
There exists a positive definite matrix $H \in \mb{S}^{d \times d}$ with eigenvalues $\lambda(H) \subseteq [\mu, L]$  such that 
\eqals{
	y = H s
}
if and only if
\eqal{  \label{exclaim}
	\nrm{}{y - \frac{L+\mu}{2}s} \leq \frac{L - \mu}{2} \nrm{}{s}.
}
\end{lem}
With this result in hand, it is easy to establish the following bounds.

\begin{lem}
\label{lemma_effect_of_lengthening}
(Choice of the Lengthening Parameter)
Suppose Assumptions \ref{ass_Lip_grad_bounded_below}, \ref{ass_bounded_noise} and \ref{ass_strong_convex} hold. Let $s \in \mb{R}^d$ be a vector such that $\nrm{}{s} \geq l$, and define $y = g(x + s) - g(x)$. If 
   $$l > 2 \epsilon_g / m,$$
then 
\eqal{
	\label{toolcon}
	\frac{y^T s}{s^T s} & \geq \bpa{m - \frac{2\epsilon_g}{l}} \eqdef \widehat m > 0\\
	\frac{y^T y}{y^T s} & \leq \bpa{M + \frac{2 \epsilon_g}{l}} \eqdef \widehat M > 0 .
}
\end{lem}   
\begin{proof}
Let $\wt{y} = \nabla \phi(x+s) - \nabla \phi(x)$. {Since $\phi \in C^2$, we have that $\nabla \phi(x+s) - \nabla \phi(x)= A s$, where $A$ is the average Hessian
\eqals{
	A = \int_0^1 ~ \nabla^2 \phi(x + t \cdot s) ~ dt .
}
 }
Since $\phi$ is $m$-strongly convex with $M$-Lipschitz continuous gradients, we know that $\lambda(A) \subseteq [m, M]$, and by Lemma~\ref{lemma_mu_L_smooth} we have
\eqal{
	\nrm{}{\wt{y} - \frac{M+m}{2} s} \leq \frac{M-m}{2}  \nrm{}{s}.
}
By \eqref{fnoise} and Assumption \ref{ass_bounded_noise}, we have
\eqals{
	\nrm{}{y - \wt{y}} \leq 2 \epsilon_g ,
}
and hence
\eqals{
	\nrm{}{y - \frac{M+m}{2} s} \leq \frac{M-m}{2}  \nrm{}{s} + 2\epsilon_g .
}
If $\nrm{}{s} \geq l$, we have
\eqals{
	\frac{M-m}{2}  \nrm{}{s} + 2\epsilon_g \leq  \frac{M-m}{2}  \nrm{}{s} + \frac{2 \epsilon_g}{l} \nrm{}{s} \cm
}
and thus
\eqals{
	\nrm{}{y - \frac{M+m}{2} s} \leq \bpa{\frac{{M}-m}{2}+\frac{2 \epsilon_g}{l}}  \nrm{}{s} .
}
By defining 
\begin{equation}   \label{amarillo}
	\widehat m  = {m - \frac{2\epsilon_g}{l}} , \qquad
	\widehat M  = {M + \frac{2 \epsilon_g}{l}} ,
\end{equation}
we have
\eqals{
	\nrm{}{y - \frac{{\widehat M}+{\widehat m}}{2} s} \leq \frac{{\widehat M}-{\widehat m}}{2}  \nrm{}{s} .
}
Note that since $l > 2\epsilon_g/m$, we have $0 < \widehat m \leq \widehat M$. By  Lemma \ref{lemma_mu_L_smooth}, we know that there exists a positive definite matrix $H$ with $\lambda(H) \subseteq [\widehat m, \widehat M]$ such that
\eqals{
	y = H s .
}
Then it immediately follows that 
\[
	\frac{y^T s}{s^T s}  \geq \widehat m \cm \quad
	\frac{y^T y}{y^T s}  \leq \widehat M,
\]
which proves the result due to \eqref{amarillo}.
\end{proof}

We thus see from this lemma that if the lengthening parameter $l$ satisfies $l >2\epsilon_g/m$, the right hand sides in  \eqref{toolcon} are strictly positive, as needed for the analysis that follows.

%%%%%%%%%%%%%%%%%%%%%%%
\subsection{Properties of the ``Good Iterates"}

We now show that the angle between the search direction of Algorithm~\ref{algo1} and the true gradient is bounded away from $90^\circ$, for a fraction of all iterates. 
We begin by stating a result from \cite[Theorem 2.1]{ByNoTool}, which describes a fundamental property of the standard BFGS method (without errors).

%%%%%%%%%%%%%%%%%%%%%%

\begin{lem}{(Existence of good iterates for classical BFGS)}
\label{lemma_good_iterates}
Let $H_0 \succ 0$, and let $\{H_k = B_k^{-1}\} $ be generated by the BFGS update \eqref{eq_bfgs_update} using \textbf{any} correction pairs $\bst{(s_k, y_k)}$ satisfying \eqref{eq_tool_paper_condition} for all $k$. Define $\Theta_k$ to be the angle between $s_k$ and $B_k s_k$, i.e.,
\eqal{
    \cos \Theta_k = \frac{s_k^T B_k s_k}{\nrm{}{s_k} \nrm{}{B_k s_k}} .
}
For a fixed scalar $q \in (0, 1)$,  let
\eqal{
	\label{eq_def_betas}
    \beta_0(q) & = \frac{1}{1-q} \bbr{{\rm tr}(B_0) - \log \det (B_0) + \wh M - 1 - \log \wh m} > 0\\
    \beta_1(q) & = e^{- \beta_0(q) /2} \in (0, 1) .\\
} 
Then we have, for all $k$,
\eqal{
	\label{eq_number_of_good_iterates}
	\left| {\Big\{j \in \bst{0, 1, \cdots, k-1} \big|\cos \Theta_j \geq \beta_1(q)\Big\}} \right| \geq q k .
}
\end{lem}

%For the rest of the paper, we pick a fixed $q \in (0, 1)$ and always use $\beta_0, \beta_1$ to represent $\beta_0(q), \beta_1(q)$ as defined in \eqref{eq_def_betas}. 

We now  establish a lower bound for the cosine of the angle between the quasi-Newton direction of Algorithm~\ref{algo1} and $-g_k$, i.e., a bound on $\cos \theta_k$ defined by setting $p_k \leftarrow - H_k g(x_k)$ in \eqref{twotheta}.

\begin{cor}
\label{cor_lower_bound_cos_theta}
Consider Algorithm~\ref{algo1} with lengthening parameter $l > 2\epsilon_g/m$ and suppose that Assumptions~\ref{ass_Lip_grad_bounded_below}, \ref{ass_bounded_noise} and \ref{ass_strong_convex} hold.  Let $\theta_k$ be the angle between $p_k = - H_k g(x_k)$ and $-g(x_k)$. 
For a given $q \in (0,1)$, set $\beta_1$ as in Lemma~\ref{lemma_good_iterates}, and define the index  $J$ of ``good iterates" generated by Algorithm~\ref{algo1} as
\eqal{
   J = \bst{j \in \mb{N} | \cos \theta_j \geq \beta_1} ,
}
as well as the set $J_k = J \cap \bst{0, 1, 2, ..., k-1}$. Then,  
\eqal{  \label{go}
   \abs{J_k} \geq q k .
}
\end{cor}
\begin{proof}
Since  $l > 2 \epsilon_g/m$, we know by \eqref{toolcon} in Lemma~\ref{lemma_effect_of_lengthening}  that conditions \eqref{eq_tool_paper_condition} are satisfied for all $k$. Since 
\eqals{
	\Theta_k = \angle \bpa{s_k, B_k s_k} = \angle \bpa{p_k, B_k p_k} = \angle \bpa{p_k, -g_k} = \theta_k,
}
\eqref{go} follows from Lemma \ref{lemma_good_iterates}.
\end{proof}

Having established a lower bound on $\cos \theta_k$ (for the good iterates), the next step is to establish a similar lower bound for $\cos \wt{\theta}_k$. To do so, we first prove the following result, which we state in some generality.

\begin{lem}
\label{lemma_bound_theta_tilde}
Let $p, g_1, g_2 \in \mb{R}^d$ be non-zero vectors. Let $\vartheta_1$ be the angle between $p$ and $g_1$, and $\vartheta_2$ the angle between $p$ and $g_2$. Assume 
\eqal{   \label{var1}
	\cos \vartheta_1 \geq \beta > 0 \cm
}
and that $g_1$ and $g_2$ satisfy 
\eqal{   \label{var2}
	\nrm{}{g_1 - g_2} \leq \epsilon .
}
If in addition
\eqal{  \label{var3}
	\frac{\epsilon}{\nrm{}{g_2}} \leq \frac{\beta}{4} ,
}
then
\eqals{
	\cos \vartheta_2 \geq \frac{\beta}{2} .
}
\end{lem}

\begin{proof}
From \eqref{var1} we have
\eqals{
	p^T g_1 \geq \beta \nrm{}{p} \nrm{}{g_1} ,
}
and by \eqref{var2}
\eqals{
	p^T g_2 \geq \nrm{}{p} \bpa{\beta \nrm{}{g_1} - \epsilon} .
}
Hence, by \eqref{var3}
\eqals{
	\cos \vartheta_2  = & \frac{p^T g_2}{\nrm{}{p} \nrm{}{g_2}} 
	%\geq \frac{\nrm{}{p} \bpa{\beta \nrm{}{g_1} - \epsilon}}{\nrm{}{p} \nrm{}{g_2}} 
	\geq \frac{{\beta \nrm{}{g_1} - \epsilon}}{\nrm{}{g_2}} \\
	\geq & \frac{\nrm{}{g_2} - \epsilon}{\nrm{}{g_2}} \beta - \frac{\epsilon}{\nrm{}{g_2}} \\
	\geq & \bpa{1 - \frac{\epsilon}{\nrm{}{g_2}}} \beta - \frac{\beta}{4} .
}
The bound \eqref{var1} implies that $\beta \leq 1$, and hence
\eqals{
	\frac{\epsilon}{\nrm{}{g_2}} \leq \frac{\beta}{4} \leq \frac{1}{4} .
}
Therefore,
\[
	\cos \vartheta_2  \geq \bpa{1 - \frac{\epsilon}{\nrm{}{g_2}}} \beta - \frac{\beta}{4} \geq \frac{\beta}{2} .
\]
\end{proof}

We also need the following well known result \cite{mybook} about the function decrease provided by the Armijo-Wolfe line search.

\begin{lem}
\label{lem_aw_one_step_descent}
Suppose $h: \mb{R}^d \to \mb{R}$ is a continuous differentiable function with an $L$-Lipschitz continuous gradient. Suppose $x \in \mb{R}^d$, and that $p \in \mb{R}^d$ is a descent direction for $h$ at $x$. Let $\theta$ be the angle between $-p$ and $\nabla h(x)$. Suppose $\alpha > 0$ is a step that satisfies the Armijo-Wolfe conditions with parameters $0 < c_1 < c_2 < 1$:
\eqal{   \label{varwolfe}
	h(x+\alpha p) & \leq h(x) + c_1 \alpha p^T \nabla h(x) \\
	p^T \nabla h(x+ \alpha p) & \geq c_2 p^T \nabla h(x) .
}
Then 
\eqals{
	h(x+\alpha p) - h(x) \leq -c_1 \frac{1-c_2}{L} \cos^2 \theta \nrm{}{\nabla h(x)}^2 .
}
\end{lem}
\begin{proof}
From the second condition in \eqref{varwolfe} we have
\eqals{
   p^T [\nabla h(x+ \alpha p) - \nabla h(x)]  \geq (c_2  - 1)p^T \nabla h(x) .
}
By Lipschitz continuity,
\eqals{
    p^T [\nabla h(x+ \alpha p) - \nabla h(x)]  \leq L \| p \|^2 \alpha \cm
}
and from this it follows that
\eqals{
	\alpha \geq -\frac{1-c_2}{L} \frac{\nabla h(x)^T p}{\nrm{}{p}^2} .
}
Substituting this into the first condition in \eqref{varwolfe} we obtain the desired result. 
\end{proof}

We can now show that a fraction of the iterates generated by Algorithm~\ref{algo1} produce a decrease in the true objective that is proportional to its gradient. We recall that the constants in the Armijo-Wolfe conditions \eqref{eq_aw_condition} satisfy
$
      0 < c_1 < c_2 <1.
$
\begin{thm}
\label{thm_decrease_per_iter}
Suppose Assumptions \ref{ass_Lip_grad_bounded_below}, \ref{ass_bounded_noise} and \ref{ass_strong_convex} are satisfied, and let $\bst{x_k}$, $\bst{p_k}$ be generated by Algorithm~\ref{algo1}. 
Define $\beta_1$ and $J$ as in {Corollary~\ref{cor_lower_bound_cos_theta}}. Choose $\delta_1, \delta_2, \wh{\delta}_1, \wh{\delta}_2 \in (0, 1)$ such that $\delta_1 + \delta_2 < c_2 - c_1$ and $ \wh{\delta}_1 < c_1, \ \wh{\delta}_2 < 1 - c_2$. If $k \in J$ and 
\eqal{
	\label{eq_conditions_ensuring_decrease_per_iter}
	\nrm{}{\nabla \phi(x_k)} \geq \max \Bigg\{A  \frac{\sqrt{{M}\epsilon_f}}{\beta_1} \cm B  \frac{\epsilon_g}{\beta_1}  \Bigg\} ,
}
where
\eqals{
	A & = \max \bst{ \sqrt{\frac{32}{(1-{c}_2+\delta_2) \delta_1}} \cm \sqrt{\frac{32}{\wh{\delta}_1 (1-c_2)}}} \\
	B & = \max \bst{\frac{4 ({c}_1 + \delta_1)}{\delta_1}  \cm \frac{2(1+{c}_2-\delta_2)}{\delta_2} \cm \frac{8}{(1-c_2)} \cm \frac{4 c_1}{\wh{\delta}_1} \cm \frac{2(1+c_2)}{\wh{\delta}_2}} \cm
}
then there exists a stepsize $\alpha_k$ which satisfies the Armijo-Wolfe conditions for $(f, g)$ with parameters $(c_1, c_2)$, i.e.,
\eqals{
	f(x_k + \alpha_k p_k) & \leq f(x_k) + c_1 \alpha_k p_k^T g(x_k) \\
	p_k^T g(x_k + \alpha_k p_k) & \geq c_2 p_k^T g(x_k) \cm
}
and \textbf{any} such stepsize also satisfies the Armijo-Wolfe conditions for $(\phi, \nabla \phi)$ with parameters $(c_1 - \wh \delta_1, c_2 + \wh \delta_2)$:
\eqals{
	\phi(x_k + \alpha_k p_k) & \leq \phi(x_k) + ({c_1 - \wh \delta_1}) \alpha_k p_k^T \nabla \phi(x_k) \\
	p_k^T \nabla \phi(x_k + \alpha_k p_k) & \geq ({c_2 + \wh \delta_2})p_k^T\nabla \phi(x_k)
}
and in addition,
\eqal{   \label{city}
	\phi(x_{k+1}) - \phi(x_k) \leq - \frac{(c_1 - \wh{\delta}_1)\bbr{1-(c_2 + \wh{\delta}_2)}\beta_1^2}{4{M}} \nrm{}{\nabla \phi(x_k)}^2 .
}
\end{thm}

\begin{proof}
Take $k \in J$. By Corollary~\ref{cor_lower_bound_cos_theta} we have that
$\cos \theta_k \geq \beta_1$. Now, by \eqref{eq_conditions_ensuring_decrease_per_iter}
\eqals{
	\nrm{}{\nabla \phi(x_k)} \geq B  \frac{\epsilon_g}{\beta_1} \geq \frac{4 ({c}_1 + \delta_1)}{\delta_1}  \frac{\epsilon_g}{\beta_1}  \geq 4  \frac{\epsilon_g}{\beta_1}  ,
}
which together with Lemma \ref{lemma_bound_theta_tilde} and {Assumption~\ref{ass_strong_convex}} implies that $\cos \wt{\theta}_k \geq \beta_1/2$. Therefore, $p_k = - H_k g(x_k)$ is a descent direction with respect to both $g(x_k)$ and $\nabla \phi(x_k)$, which will enable us to apply 
Theorems \ref{thm_existence_of_aw_step} and \ref{thm_goodness_of_aw_step}. 

Before doing so, we need to verify that the assumptions of those two theorems are satisfied, namely \eqref{who} and \eqref{tennis}.
To see this, note that since we have shown that
\eqals{
	\cos \theta_k \geq \beta_1, \quad \cos \wt{\theta}_k \geq \frac{\beta_1}{2}
}
then from \eqref{eq_conditions_ensuring_decrease_per_iter} it follows that
%imply \eqref{who} and \eqref{tennis}. 
\eqals{
& \nrm{}{\nabla \phi(x_k)}\\
\geq ~ &\max \Bigg\{A  \frac{\sqrt{{M}\epsilon_f}}{\beta_1} \cm B  \frac{\epsilon_g}{\beta_1}  \Bigg\} \\
\geq ~ &\max \bst{\frac{4 ({c}_1 + \delta_1)\epsilon_g}{\delta_1 \beta_1} \cm \frac{2(1+{c}_2-\delta_2)\epsilon_f}{\delta_2 \beta_1} \cm \sqrt{\frac{32 M \epsilon_f}{(1-{c}_2+\delta_2) \delta_1 \beta_1^2}}
}\\
\geq ~ & \max \bst{\frac{4 ({c}_1 + \delta_1)\epsilon_g}{\delta_1 \cos \theta_k} \cm \frac{2(1+{c}_2-\delta_2)\epsilon_f}{\delta_2 \cos \theta_k} \cm \sqrt{\frac{16 M \epsilon_f}{(1-{c}_2+\delta_2) \delta_1 \cos \theta_k \cos \wt{\theta}_k}}} ,
}
as well as
\eqals{
	&\nrm{}{\nabla \phi(x_k)} \\
	\geq ~ & \max \Bigg\{A  \frac{\sqrt{{M}\epsilon_f}}{\beta_1} \cm B  \frac{\epsilon_g}{\beta_1}  \Bigg\} \\
	\geq ~ & \max \bst{\frac{8\epsilon_g}{(1-c_2)\beta_1} \cm \frac{4 c_1\epsilon_g}{\wh{\delta}_1\beta_1} \cm \frac{2(1+c_2)\epsilon_g}{\wh{\delta}_2\beta_1} \cm \sqrt{\frac{32 M \epsilon_f}{\wh{\delta}_1 (1-c_2) \beta_1^2}}} \\
	\geq ~ &\max \bst{\frac{8\epsilon_g}{(1-c_2)\cos \theta_k} \cm \frac{2 c_1\epsilon_g}{\wh{\delta}_1 \cos \wt{\theta}_k} \cm \frac{(1+c_2)\epsilon_g}{\wh{\delta}_2 \cos \wt{\theta_k}} \cm \sqrt{\frac{16 M \epsilon_f}{\wh{\delta}_1 (1-c_2) \cos \theta_k \cos \wt{\theta_k}}}} .
}

Therefore, by Theorems \ref{thm_existence_of_aw_step} and \ref{thm_goodness_of_aw_step} there exists a stepsize $\alpha_k$ which satisfies the Armijo-Wolfe conditions for $(f, g)$ with parameters $(c_1, c_2)$, and such $\alpha_k$ also satisfies the Armijo-Wolfe conditions for $(\phi, \nabla \phi)$ with parameters $(c_1 - \wh{\delta}_1, c_2 + \wh{\delta}_2)$. 
We then apply Lemma \ref{lem_aw_one_step_descent} with $h(\cdot) \gets \phi(\cdot)$, $\theta \gets \wt{\theta}_k$ and $L \gets M$, Armijo-Wolfe parameters $(c_1 - \wh \delta_1, c_2 + \wh \delta_2)$ to obtain
%\begin{equation}
	\begin{align*}
		\phi(x_{k+1}) - \phi(x_k)& \leq - \frac{(c_1 - \wh{\delta}_1)\bbr{1-(c_2 + \wh{\delta}_2)}}{{M}} ~ \cos^2 \wt{\theta}_k \nrm{}{\nabla \phi(x_k)}^2  \\
	& \leq - \frac{(c_1 - \wh{\delta}_1)\bbr{1-(c_2 + \wh{\delta}_2)}\beta_1^2}{4{M}} \nrm{}{\nabla \phi(x_k)}^2 .
	\end{align*}
%\end{equation}
\end{proof}

The constants $A, B$, as well as the rate constant in \eqref{city}, do not depend on the objective function or the noise level, but only on the parameters $c_1, c_2$. There is, nevertheless, some freedom in the specification of $A, B$ and the constant in \eqref{city} through the choices of  $\delta_1, \delta_2, \wh \delta_1, \wh \delta_2$. From now on, we make a specific choice for the latter four constants, which simplifies Theorem~\ref{thm_decrease_per_iter}, as shown next.

\begin{cor}
\label{cor_decrease_per_iter_for_specfic_deltas}
Suppose Assumptions \ref{ass_Lip_grad_bounded_below}, \ref{ass_bounded_noise} and \ref{ass_strong_convex} are satisfied, and let $\bst{x_k}$ be generated by Algorithm \ref{algo1}. Choose $\delta_1, \delta_2, \wh \delta_1, \wh \delta_2$ as 
\eqal{
	\label{eq_value_for_deltas}
	\delta_1  = \frac{c_2-c_1}{4} \cm \ \delta_2  = \frac{c_2-c_1}{4} \cm \ \wh \delta_1 & = \frac{c_1}{2} \cm \ \wh \delta_2  = \frac{1-c_2}{2}	.
} 
If $k \in J$ and 
\eqals{
	\nrm{}{\nabla \phi(x_k)} \geq \max \Bigg\{ &A  \frac{\sqrt{{M}\epsilon_f}}{\beta_1} \cm B  \frac{\epsilon_g}{\beta_1}  \Bigg\} \cm
}
where
\eqal{
	\label{eq_def_for_A_and_B_with_specific_deltas}
	A & = \max \left\{\frac{16 \sqrt{2}}{\sqrt{(c_2-c_1) (4-c_1-3 c_2)}},\frac{8}{\sqrt{c_1(1-c_2)}}\right\} \\
	B & = \max \left\{\frac{8}{1-c_2},\frac{8 (1+c_1)}{c_2-c_1}+6\right\} ,
}
then there exists a stepsize $\alpha_k$ which satisfies the Armijo-Wolfe conditions on $(f, g)$ with parameters $(c_1, c_2)$, i.e.,
\eqals{
	f(x_k + \alpha_k p_k) & \leq f(x_k) + c_1 \alpha_k p_k^T g(x_k) \\
	p_k^T g(x_k + \alpha_k p_k) & \geq c_2 p_k^T g(x_k) ,
}
and \textbf{any} such stepsize also satisfies the Armijo-Wolfe conditions on $(\phi, \nabla \phi)$ with parameters $(c_1/2, c_2/2 + 1)$:
\eqals{
	\phi(x_k + \alpha_k p_k) & \leq \phi(x_k) + \frac{c_1}{2} \alpha_k p_k^T \nabla \phi(x_k) \\
	p_k^T \nabla \phi(x_k + \alpha_k p_k) & \geq \frac{1+c_2}{2} p_k^T\nabla \phi(x_k) ,
}
and in addition,
\eqals{
	\phi(x_{k+1}) - \phi(x_k) \leq -\frac{c_1(1-c_2) \beta_1^2}{16 M} \nrm{}{\nabla \phi(x_k)}^2 .
}
\end{cor}
\begin{proof}
We begin by verifying that the choices \eqref{eq_value_for_deltas} of $\delta_1, \delta_2, \wh \delta_1, \wh \delta_2$ satisfy the requirements in Theorem \ref{thm_decrease_per_iter}. It is clear that $\delta_1, \delta_2, \wh \delta_1, \wh \delta_2 \in (0, 1)$ since $0 < c_1 <c_2 < 1$. We also have
\eqals{
	\delta_1 + \delta_2 = \frac{c_2 - c_1}{2} < c_2 - c_1 \cm \quad  \wh{\delta}_1 = \frac{c_1}{2} < c_1 \cm \quad  \wh{\delta}_2 = \frac{1-c_2}{2} < 1 - c_2 .
}
Applying Theorem \ref{thm_decrease_per_iter} with the choices  \eqref{eq_value_for_deltas}, we have
\eqals{
	A & = \max \bst{ \sqrt{\frac{32}{(1-{c}_2+\delta_2) \delta_1}} \cm \sqrt{\frac{32}{\wh{\delta}_1 (1-c_2)}}} \\
	& = \max \left\{\frac{16 \sqrt{2}}{\sqrt{(c_2-c_1) (4-c_1-3 c_2)}},\frac{8}{\sqrt{c_1(1-c_2)}}\right\} \\
	& ~ \\
	B & = \max \bst{\frac{4 ({c}_1 + \delta_1)}{\delta_1}  \cm \frac{2(1+{c}_2-\delta_2)}{\delta_2} \cm \frac{8}{(1-c_2)} \cm \frac{4 c_1}{\wh{\delta}_1} \cm \frac{2(1+c_2)}{\wh{\delta}_2}}\\
	& = \max \left\{\frac{8}{1-c_2},\frac{8 (1+c_1)}{c_2-c_1}+6\right\}.
}
Therefore, by Theorem \ref{thm_decrease_per_iter} we know that there exists a stepsize $\alpha_k$ which satisfies the Armijo-Wolfe conditions for $(f, g)$ with parameters $(c_1, c_2)$, and any such stepsize also satisfies the Armijo-Wolfe conditions for $(\phi, \nabla \phi)$ with parameters $(c_1 - \wh \delta_1, c_2 + \wh \delta_2) = (c_1/2, c_2/2 + 1)$. In addition, we also have
\begin{align*}
	\phi(x_{k+1}) - \phi(x_k) & \leq - \frac{(c_1 - \wh{\delta}_1)\bbr{1-(c_2 + \wh{\delta}_2)}\beta_1^2}{4{M}} \nrm{}{\nabla \phi(x_k)}^2 \\
	& = -\frac{c_1(1-c_2) \beta_1^2}{16 M} \nrm{}{\nabla \phi(x_k)}^2 .
\end{align*}
\end{proof}

\subsection{Convergence Results}
We are ready to state the main convergence results for our algorithm, which is simply  Algorithm~\ref{algo1} using a lengthening parameter $l$ such that 
\begin{equation}  \label{long}
      l >2\epsilon_g/{m},
\end{equation}
where $\epsilon_g$ is the maximum error in the gradient and $m$ is the strong convexity parameter.  Although knowledge of these two constants may not always be available in practice,  there are various procedures for estimating them, as discussed in Section~\ref{sec:numerical}.

We begin by establishing some monotonicity results for the true objective function $\phi$. Note that since Algorithm~\ref{algo1} either computes a zero step (when $\alpha^*=0$) or generates a new iterate that satisfies the Armijo decrease \eqref{eq_aw_condition_noised}, the sequence $\{f(x_k)\}$ is non-increasing. 
\begin{thm}
\label{thm_min_phi_k}
Suppose Assumption \ref{ass_bounded_noise} is satisfied, and let $\bst{x_k}$ be generated by Algorithm~\ref{algo1} with $l$ satisfying \eqref{long}. Define
\eqal{   \label{french}
	\xi_k = \min_{i \in [k]} ~ \phi(x_i) \cm \quad \text{where} ~ [k] \eqdef ~ \bst{i \in \mb{N}| 0 \leq i \leq k}.
}
Then $\bst{\xi_k}$ is non-increasing
%, i.e.,
%$
%	\xi_{k+1}  \leq \xi_k \cm 
%$
and 
\eqals{
	\xi_k \leq \phi(x_k) \leq \xi_k + 2 \epsilon_f \cm \  \forall k \in \mb{N} .
}
\end{thm}
\begin{proof}
By definition, $\bst{\xi_j}$  forms a non-increasing sequence, and we noted above that $\bst{f(x_k)}$ is also non-increasing and therefore 
\eqals{
	f(x_{j}) = \min_{i \in [j]} ~ f(x_i) .%\cm \text{where} ~ [j] \overset{\text{def}}{=} ~ \bst{i \in \mb{N}^+| 0 \leq i \leq j}
}
By Assumption~\ref{ass_bounded_noise} we have
\eqals{
	f(x_i) \leq \phi(x_i) + \epsilon_f .
}
Hence
\eqals{
	f(x_{j}) = \min_{i \in [j]} ~ f(x_i) \leq \min_{i \in [j]} ~ \bpa{\phi(x_i) +  \epsilon_f} = \min_{i \in [j]} ~ \phi(x_i) +  \epsilon_f ,
}
and recalling again Assumption~\ref{ass_bounded_noise}, we have
\eqals{
	\phi(x_j) \leq f(x_j) + \epsilon_f \leq \min_{i \in [j]} ~ \phi(x_i) +  2\epsilon_f .
}
Since
\eqals{
	\xi_j= \min_{i \in [j]} ~ \phi(x_i) \leq \phi(x_j) \cm
}
we conclude that 
\[
	\xi_j \leq \phi(x_j) \leq \xi_j + 2\epsilon_f .
\]

\end{proof}

The next result shows that, before the iterates $\bst{x_k}$ reach a neighborhood of the solution where the error dominates, the sequence $\bst{\phi(x_k)- \phi^*}$ converges to the value $2 \epsilon_f$  at an R-linear rate. Here $\phi^*$ denotes the optimal value of $\phi$.

\begin{thm}
\label{thm_r_linear_convergence_before_hitting_N1}
{\rm [Linear Convergence to $\mathcal{N}_1$]}
Suppose Assumptions~\ref{ass_Lip_grad_bounded_below}, \ref{ass_bounded_noise} and \ref{ass_strong_convex} are satisfied, and  let $\bst{x_k}$ be generated by Algorithm~\ref{algo1} with the choice \eqref{long}. Let
\eqals{
	\mathcal{N}_1 = \bst{x \Bigg|\nrm{}{\nabla \phi(x)} \leq \max \Bigg\{ A \frac{\sqrt{{M}\epsilon_f}}{\beta_1} \cm B  \frac{\epsilon_g}{\beta_1}  \Bigg\}},
}
where $A, B$ are given in \eqref{eq_def_for_A_and_B_with_specific_deltas}. 
Let
\eqals{
	K = \min_{k} ~ \bst{k \in \mb{N}|x_k \in \mathcal{N}_1}
}
be the index of the first iterate that enters $\mathcal{N}_1$ (we define $K = +\infty$ if no such iterate exists). Then there exists $\rho \in (0, 1)$ such that
\eqals{
	\phi(x_k) - \phi^* \leq \rho^k ~ (\phi(x_0) - \phi^*) + 2\epsilon_f \cm \ \forall k \leq K -1 .
}
\end{thm}
\begin{proof}
By definition, we have that $ \forall k \leq K-1$ 
\eqal{   \label{ha}
	\nrm{}{\nabla \phi(x_k)} > \max \Bigg\{ A  \frac{\sqrt{{M}\epsilon_f}}{\beta_1} \cm B  \frac{\epsilon_g}{\beta_1}  \Bigg\} .
}
Choose $0 \leq j \leq k \leq K-1$, and let $J$ be as defined in Corollary~\ref{cor_lower_bound_cos_theta}. {If} ${ j \in J}$, then by Corollary~\ref{cor_decrease_per_iter_for_specfic_deltas} we have
\eqals{
	\phi(x_{j+1}) - \phi(x_j) \leq -\zeta \nrm{}{\nabla \phi(x_j)}^2
}
where 
\eqals{
	\zeta =  \frac{c_1(1-c_2)\beta_1^2}{16{M}} .
}
By Theorem~\ref{thm_min_phi_k}, we have that $\phi(x_j) \leq \xi_j + 2 \epsilon_f$, and hence
\eqals{
	\phi(x_{j+1}) \leq \xi_j + 2\epsilon_f -\zeta \nrm{}{\nabla \phi(x_j)}^2 .
}
Recalling that
\eqals{
	A = \max \left\{\frac{16 \sqrt{2}}{\sqrt{(c_2-c_1) (4-c_1-3 c_2)}},\frac{8}{\sqrt{c_1(1-c_2)}}\right\}, 
}
and by \eqref{ha} we have
\eqals{
	 \zeta \nrm{}{\nabla \phi(x_j)}^2 \geq ~ & \frac{c_1(1-c_2)}{16}  A^2  \epsilon_f\\
	\geq ~ & \frac{c_1(1-c_2)}{16}  \bbr{\frac{8}{\sqrt{c_1(1-c_2)}}}^2  \epsilon_f \\
	 = ~ & 4 \epsilon_f ,
}
and thus
\eqals{
	\phi(x_{j+1}) \leq \xi_j  -\frac{\zeta}{2} \nrm{}{\nabla \phi(x_j)}^2 .
}
Since $\phi$ is strongly convex by Assumption~\ref{ass_strong_convex}, we have
\eqals{
	\nrm{}{\nabla \phi(x_j)}^2 \geq 2 {m} (\phi(x_j) - \phi^*) \geq 2 {m}(\xi_j - \phi^*) \cm
}
% Hence
% \eqals{
% 	\nrm{}{\nabla \phi(x_j)}^2 \geq 2 \wt{m} (\phi(x_j) - \phi^*) \geq 2 \wt{m}(\xi_j - \phi^*)
% }
thus we have
\eqals{
	\xi_{j+1} \leq \phi(x_{j+1}) \leq \xi_j  -\frac{\zeta}{2} \nrm{}{\nabla \phi(x_j)}^2 \leq \xi_j - {m}\zeta(\xi_j - \phi^*)
}
i.e.,
\eqals{
	\xi_{j+1} -\phi^* \leq (1 - {m}\zeta)(\xi_j - \phi^*) .
}

The relation above holds  if $j \in J$.  If $j \notin J$, all we can ascertain is that
\eqals{
	\xi_{j+1} \leq \xi_j .
}
% By lemma \ref{lemma_good_iterates}, we have $\abs{[k-1]\cap J} \geq pk$, hence
By Corollary \ref{cor_lower_bound_cos_theta}, we have $\abs{[k-1]\cap J} \geq qk$, hence
\eqals{
	\xi_k - \phi^* \leq (1-{m}\zeta)^{qk} ~ (\xi_0 - \phi^*) = \rho^k (\phi(x_0) - \phi^*)
}
where $\rho = (1-{m}\zeta)^q$. Since $\phi(x_k) \leq \xi_k + 2\epsilon_f$, we have
\[
	\phi(x_k) - \phi^* \leq \rho^k (\phi(x_0) - \phi^*)	+ 2\epsilon_f.
\]
\end{proof}

The next result shows that the iterates generated by the algorithm enter the neighborhood ${\cal N}_1$ in a finite number of iterations.

\begin{thm}
\label{thm_finite_hitting_time_for_N1}
Suppose Assumptions \ref{ass_Lip_grad_bounded_below}, \ref{ass_bounded_noise} and \ref{ass_strong_convex} are satisfied. Let $\bst{x_k}$ be generated by Algorithm~\ref{algo1} using \eqref{long}. Let $\mathcal{N}_1$ and $K$ be defined as in Theorem~\ref{thm_r_linear_convergence_before_hitting_N1}. If  in addition we assume that $\max \bst{\epsilon_f, \epsilon_g} > 0$, then we have 
\eqals{
	K < +\infty
}
\end{thm}

\begin{proof}
Suppose, by the way of contradiction, that $K = +\infty$, i.e., that $x_k \notin {\cal N}_1$, for all $k$. Pick arbitrary $\delta > 0$, then by Theorem~\ref{thm_r_linear_convergence_before_hitting_N1} we have
\eqals{
	\phi(x_k) - \phi^* \leq \delta + 2\epsilon_f \cm
}
for sufficiently large $k$. On the other hand, by Assumption~\ref{ass_Lip_grad_bounded_below}, 
\eqals{
	\nrm{}{\nabla \phi(x)}^2 \leq 2{M}(\phi(x) - \phi^*) \cm \forall x \in \mb{R}^d.
}
Hence,
\eqals{
	\nrm{}{\nabla \phi(x_k)}^2 \leq 4 {M} \epsilon_f + 2 {M}\delta .
}
Choose $\delta$ sufficiently small such that
\eqals{
	\nrm{}{\nabla \phi(x_k)}^2 \leq 4 {M} \epsilon_f + 2 {M}\delta \leq \bbr{\max \Bigg\{ A  \frac{\sqrt{{M}\epsilon_f}}{\beta_1} \cm B  \frac{\epsilon_g}{\beta_1}  \Bigg\}}^2 ,
}
which is always possible since $A > 2$ and $\beta_1 \in (0, 1)$. Therefore, $x_k \in \mathcal{N}_1$ yielding a contradiction.
\end{proof}

The next result shows that after an iterate has entered the neighborhood ${\cal N}_1$, all subsequent iterates cannot stray too far away from the solution in the sense that their function values remain within a band of width $2 \epsilon_f$ of the {largest} function value obtained inside ${\cal N}_1$.

\begin{thm}
Suppose Assumptions \ref{ass_Lip_grad_bounded_below}, \ref{ass_bounded_noise} and \ref{ass_strong_convex} are satisfied. Let $\bst{x_k}$ be generated by Algorithm~\ref{algo1} with the choice \eqref{long}. Let $\mathcal{N}_1$ and $K$ be defined as in Theorem~\ref{thm_r_linear_convergence_before_hitting_N1}, and let
\eqals{
	\wh{\phi} = \max_{x \in \mathcal{N}_1} ~ \phi(x) ,
}
and 
\eqals{
	\mathcal{N}_2 = \bst{x|\phi(x) \leq \wh{\phi} + 2\epsilon_f} .
}
Then, 
\eqals{
	x_k \in \mathcal{N}_2 \cm \ \forall k \geq K .
}
\end{thm}
\begin{proof}
Since $\phi$ is twice continuously differentiable and strongly convex, $\mathcal{N}_1$ defined in Theorem \ref{thm_r_linear_convergence_before_hitting_N1} is a compact set, so $\wh{\phi}$ is well-defined.
By Theorem \ref{thm_finite_hitting_time_for_N1}, $K < \infty$. Choose any $k \geq K$. Since $x_K \in \mathcal{N}_1$ and $k \geq K$, we have
\eqals{
	\xi_k \leq \xi_K \leq \phi(x_K) \leq \wh{\phi} .
}
Recalling Theorem~\ref{thm_min_phi_k},  
\eqals{
	\phi(x_k) \leq \xi_k + 2 \epsilon_f \leq \wh{\phi} + 2\epsilon_f
}
which shows that $x_k \in \mathcal{N}_2$.
\end{proof}

Finally, we have the following result regarding the lengthening operation. It shows that for all ``good iterates"  that are sufficiently away from ${\cal N}_1$ lengthening is not necessary. 
%\st{Since most iterates are good iterates, this means that lengthening  will rarely be needed while the algorithm is operating outside the area dominated by noise.}
\begin{thm}
\label{thm_lengthening_for_good_iterates}
Suppose Assumptions \ref{ass_Lip_grad_bounded_below}, \ref{ass_bounded_noise} and \ref{ass_strong_convex} are satisfied. Let $\bst{x_k}$ be generated by Algorithm \ref{algo1} with lengthening parameter $l$ satisfying \eqref{long}. Let $J$ be defined as in Corollary~\ref{cor_lower_bound_cos_theta}, and $A,B$ be defined as \eqref{eq_def_for_A_and_B_with_specific_deltas}. 
If $k \in J$ and
% \eqals{
% 	\nrm{}{\nabla \phi(x_k)} \geq \max \Bigg\{ A  \frac{\sqrt{{M}\epsilon_f}}{\beta_1} \cm B \frac{\epsilon_g}{\beta_1} \cm \frac{2 l {M}}{(1-c_2 - \wh{\delta}_2) \beta_1}  \Bigg\} ,
% }
\eqals{
	\nrm{}{\nabla \phi(x_k)} \geq \max \Bigg\{ A  \frac{\sqrt{{M}\epsilon_f}}{\beta_1} \cm B \frac{\epsilon_g}{\beta_1} \cm \frac{4 l {M}}{(1-c_2) \beta_1}  \Bigg\} ,
}
then  $\nrm{}{\alpha_k p_k} \geq l$, meaning that step~\ref{algo_line_lengthening} of Algorithm~\ref{algo1} is not executed.
\end{thm}

\begin{proof}
Since $k \in J$ and
\eqals{
	\nrm{}{\nabla \phi(x_k)} \geq \max \Bigg\{ A  \frac{\sqrt{{M}\epsilon_f}}{\beta_1} \cm B  \frac{\epsilon_g}{\beta_1} \Bigg\} ,
}
by Theorem~\ref{thm_decrease_per_iter} and Corollary~\ref{cor_decrease_per_iter_for_specfic_deltas} we know that the stepsize $\alpha_k$ satisfies 
% \eqals{
% 	\phi(x_k + \alpha_k p_k) & \leq \phi(x_k) + (c_1 - \wh{\delta}_1) \alpha_k p_k^T \nabla \phi(x_k) \\
% 	p_k^T \nabla \phi(x_k + \alpha_k p_k) & \geq (c_2 + \wh{\delta}_2) p_k^T \nabla \phi(x_k)
% }
\eqals{
	\phi(x_k + \alpha_k p_k) & \leq \phi(x_k) + \frac{c_1}{2} \alpha_k p_k^T \nabla \phi(x_k) \\
	p_k^T \nabla \phi(x_k + \alpha_k p_k) & \geq \frac{1+c_2}{2} p_k^T \nabla \phi(x_k).
}
Thus we have a lower bound on $\alpha_k$:
\eqals{
	\alpha_k \geq -\frac{1-c_2}{2{M}} \frac{\nabla \phi(x_k)^T p_k}{\nrm{}{p_k}^2}.
}
Then we have
\eqals{
	 \nrm{}{\alpha_k p_k} 
	 \geq ~ &\frac{1-c_2}{2{M}} {\nrm{}{\nabla \phi(x_k)}} \cos \wt{\theta}_k \\
	 \geq ~ &\frac{(1-c_2)\beta_1}{4{M}} {\nrm{}{\nabla \phi(x_k)}} .
}
Since 
\eqals{
	\nrm{}{\nabla \phi(x_k)} \geq \frac{4 l {M}}{(1-c_2) \beta_1} ,
}
% \eqals{
% 	\nrm{}{\nabla \phi(x_k)} \geq \frac{2 l {M}}{(1-c_2 - \wh{\delta}_2) \beta_1}
% }
we have
\[
	\nrm{}{s_k}  \geq  \frac{(1-c_2)\beta_1}{4{M}} {\nrm{}{\nabla \phi(x_k)}} \geq l .
\] 

\end{proof}

\section{Numerical Experiments}   \label{sec:numerical}
We implemented Algorithm \ref{algo1} and tested it on a $4$-dimensional quadratic function of the form
\eqals{
	\phi(x) = \frac{1}{2} x^T T x ,
}
where the eigenvalues of $T$ are
$
	\lambda(T) = \bst{10^{-2}, 1, 10^2, 10^4}. 
$
Thus, the strong convexity parameter is $m=10^{-2}$ and the Lipschitz constant $M=10^4$.
%the condition number of the Hessian is $\kappa(T) = 10^6$.

The noise in the function $\epsilon(x)$ was computed by uniformly sampling from the interval $[-\epsilon_f, \epsilon_f]$, and the noise in the gradient $e(x)$ by uniformly sampling from the closed ball 
%$B(0; \epsilon_g) = \bst{x \in \mb{R}^d|\nrm{}{x} \leq \epsilon_g}$.
$\nrm{2}{x} \leq \epsilon_g$. The maximum noise (or error) level was chosen as $\epsilon_g = \epsilon_f = 1$. We computed the lengthening parameter $l$ in Algorithm~\ref{algo1} as $l = 4\epsilon_g/m$, which is twice as large as the lower bound stipulated in Lemma \ref{lemma_effect_of_lengthening}. 
%\eqals{
%	l = \frac{4\epsilon_g}{\wt{m}} > \frac{2\epsilon_g}{\wt{m}}
%}
%In this way, it follows from Lemma \ref{lemma_effect_of_lengthening} that $m = \wt{m}/2$. 

The line search implements the standard bisection Armijo-Wolfe  search with parameters $c_1 = 0.01 , c_2 = 0.5$. If the line
search is unable to find an acceptable stepsize within 
%$$\texttt{max\_line\_search\_iterations} = 64$ 
64 iterations, its is considered to have failed, and we set $\alpha_k = 0$. Algorithm~\ref{algo1} terminates if: i)  $\nrm{}{g(x_k)} \leq  10^{-5}$; or b) 30 consecutive line search failures occur;
%$\texttt{max\_line\_search\_failures} = 30$ consecutive line search failures are encountered, 
c) or if Algorithm~\ref{algo1} reaches the limit of {60} iterations.
%$\texttt{max\_optimization\_iterations} = 80$ total iterations are reached. 
The initial iterate is $x_0 = 10^5 \cdot (1,1,1,1)^T$ for which $\nrm{}{\nabla \phi(x_0)} \approx 10^9$. 

Figures \ref{fig_function} and \ref{fig_gradient_norm} plot the results of $20$ runs of Algorithm \ref{algo1}, all initialized at the vector $x_0$ given above. In both figures, we indicate the first iteration (in all runs) when the differencing interval was lengthened, i.e., when step~\ref{algo_line_lengthening} of  Algorithm \ref{algo1} was executed. We observe from Figure~\ref{fig_function} that  Algorithm \ref{algo1} quickly drives the optimality gap $\phi(x_k) - \phi^*$ to the noise level. Figure~\ref{fig_cond_number} plots the log of the condition number of the matrix $H_k^{1/2} \nabla^2 \phi(x_k) H_k^{1/2}$ against the  iteration number $k$. For this small dimensional quadratic, the BFGS approximation converges to the true Hessian when errors are not present.  Figure~\ref{fig_cond_number} shows that the Hessian approximation does not deteriorate after the iterates enter the region where noise dominates, illustrating the benefits of the lengthening strategy.

% (The first occurrence of lengthening is $k = 8$, i.e., the $9$-th iteration.)]}

\begin{figure}[htp!]
	\centering
	\includegraphics[width=.6\textwidth]{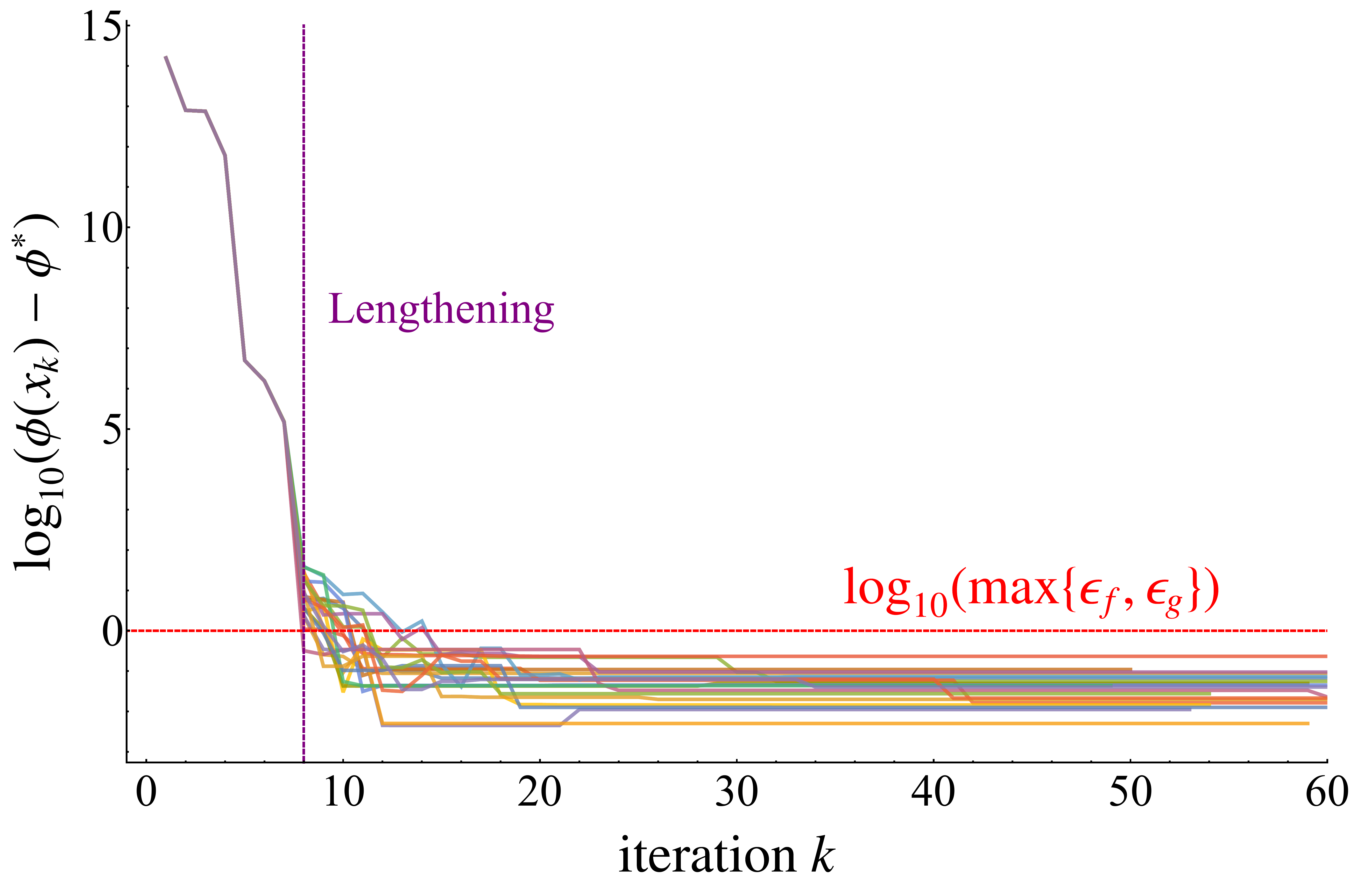}
	\caption{Results of 20 runs of Algorithm \ref{algo1}. The graph plots the log of the optimality gap for the true function, $\log_{10}\bpa{\phi(x_k) - \phi^*}$, against the iteration number $k$. The horizontal red dashed line corresponds to the noise level $\log_{10}~\max\bst{\epsilon_g, \epsilon_f}=0$. The vertical purple dashed line marks the first iteration at which lengthening is performed ($k = 8$). 
%	\jn{[I think that we can trim both figures at 100 iterations so that they fit side by side]} \yx{[Trimmed to 80 iterations, though it seems 60 is enough]}
	}
	\label{fig_function}
\end{figure}

\begin{figure}[htp!]
	\centering
	\includegraphics[width=.6\textwidth]{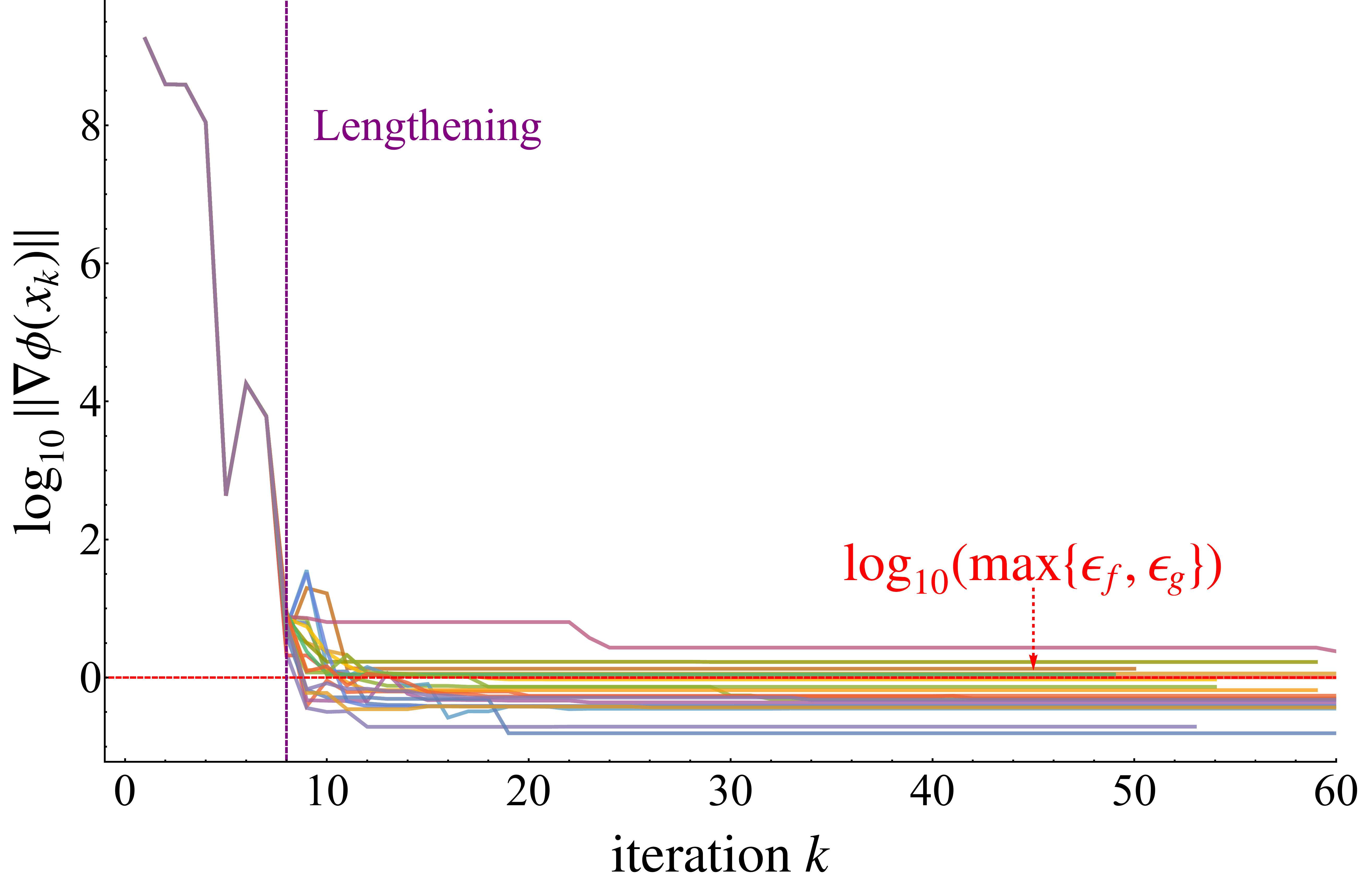}
	\caption{Log of the norm of true gradient $\log_{10} \nrm{}{\nabla \phi(x_k)}$ against iteration $k$ for 20 runs of Algorithm~\ref{algo1}. The horizontal red dashed line corresponds to the noise level, and the vertical purple dashed line corresponds to the first iteration at which lengthening is performed.}
	\label{fig_gradient_norm}
\end{figure}

\begin{figure}[htp!]
	\centering
	\includegraphics[width=.6\textwidth]{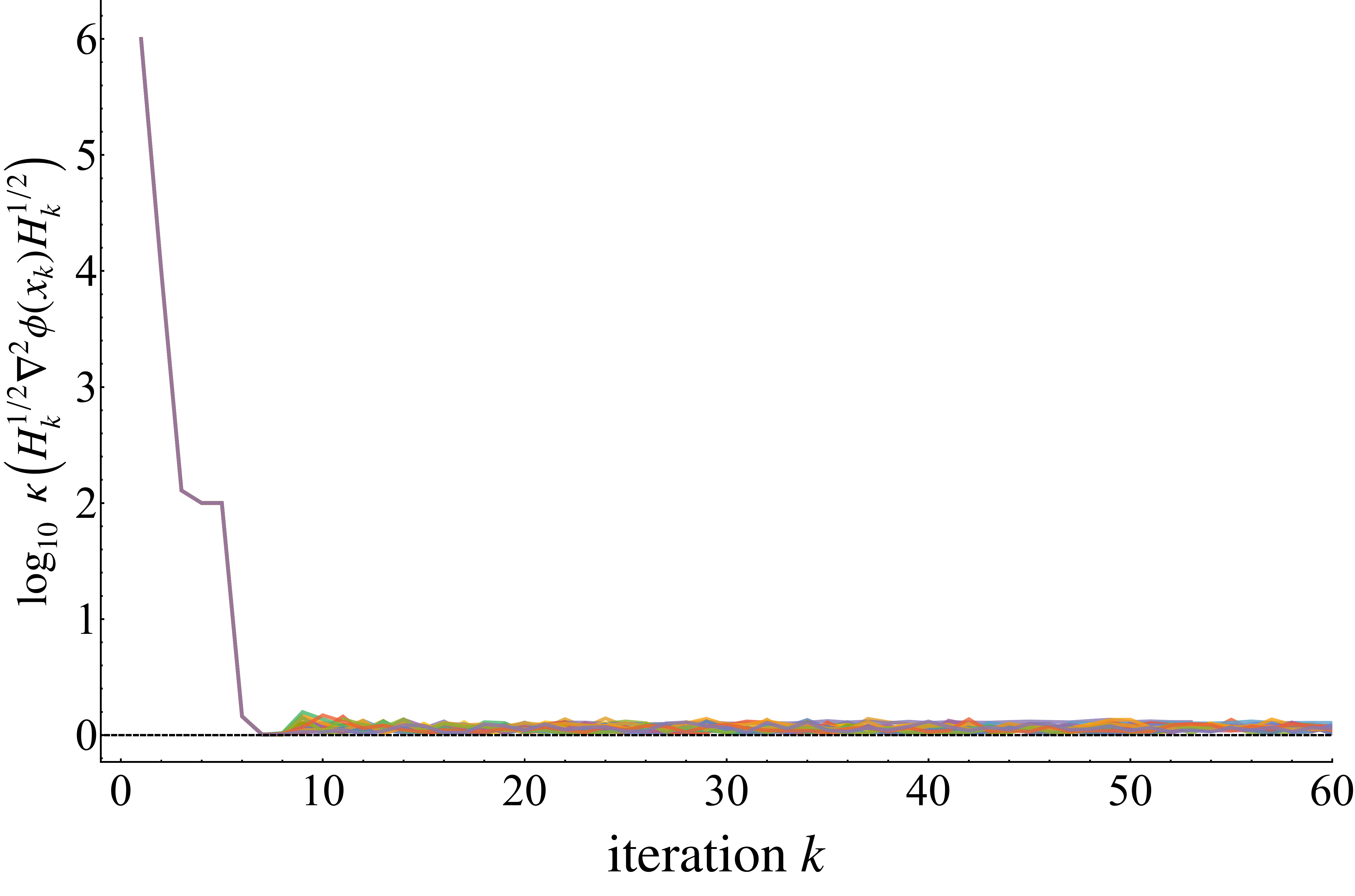}
	\caption{Log of the condition number of $H_k^{1/2} \nabla^2 \phi(x_k) H_k^{1/2}$ against iteration $k$. Note that after the iteration reaches the noise level, the Hessian approximation remains accurate.}
	\label{fig_cond_number}
\end{figure}

% these are added to ensure the figures are in-place, for easier editing
%\newpage~
%\newpage~
%\newpage~
%\newpage
% should be removed in the end

\newpage
\section{Final Remarks}  \label{sec:final}
In this paper, we analyzed the BFGS method when the function and gradient evaluations contain errors. We do not assume that errors diminish as the iterates converge to the solution, or that the user is able to control the magnitude of the errors at will; instead we consider the case when errors are always present. Because of this, our analysis focuses on global linear convergence to a neighborhood of the solution, and not on conditions that ensure superlinear convergence --- something that would require errors to diminish very rapidly.

In the regime where the gradient $\| \nabla \phi(x) \| $ of the objective function is sufficiently larger than the errors, we would hope for the BFGS method to perform well. However, even in that setting, errors can contaminate the Hessian update, and the line {search} can give conflicting information. Nevertheless, we show that a simple modification of the BFGS method inherits the good performance of the classical method (without errors). In particular, we  extend one of the hallmark results of BFGS, namely Theorem~2.1 in \cite{ByNoTool}, which shows that under mild conditions a large fraction of the BFGS iterates are good iterates, meaning that they do not tend to be orthogonal to the gradient. We also establish conditions under which an Armijo-Wolfe line search on the noisy function yields sufficient decrease in the true objective function. These two results are then combined to establish global convergence. 

The modification of the BFGS method proposed here consists of ensuring that the length of the interval used to compute gradient differences is large enough so that differencing is stable. Specifically, if the line search indicates that the size of the latest step is not large enough compared to the size the error, then the corrections pairs $(s_k, y_k)$ used to update the BFGS matrix are modified. Instead of using $s_k$ as the differencing interval, we lengthen it  and compute gradient differences based on the end points of the elongated interval. This allows us to establish convergence results to a neighborhood of the solution where progress is not possible, along the lines of Nedic and Bertsekas \cite{nedic2001convergence}. An additional feature of our modified BFGS method is that,  when the iterates enter the region where errors dominate, the Hessian approximation does not get corrupted.

The numerical results presented here are designed to verify only the behavior predicted by the theory. In our implementation of Algorithm~\ref{algo1}, we assume that the size of the errors and the strong convexity parameter are known, as this helps us determine the size of the lengthening parameter $l$. In a separate paper, we will consider a practical implementation of our algorithm that estimates $l$ adaptively,  that is able to deal with nonconvexity, and that provides a limited memory version of the algorithm.  We believe that the theory presented in this paper will be  useful in the design of such a practical algorithm.

\section{Appendix A}

\begin{proof}[Proof of Lemma \ref{lemma_mu_L_smooth}]~~~

\medskip\noindent
\textbf{Part I.}  We first show that if that $y=Hs$ with  $\lambda(H) \subseteq [\mu, L]$ then \eqref{exclaim} holds.
Clearly,
\eqals{
	\lambda \bpa{H - \frac{L + \mu}{2} I} \subseteq \bbr{-\frac{L-\mu}{2}, \frac{L-\mu}{2}} .
}
Since $H - (L+\mu)I/2$ is symmetric, we have
\eqals{
	\nrm{}{H - \frac{L + \mu}{2} I} \leq  \frac{L-\mu}{2} .
}
Since
\eqals{
	y - \frac{L+\mu}{2} s = \bpa{H - \frac{L + \mu}{2} I} s \cm
}
we conclude that
\[
	\nrm{}{y - \frac{L+\mu}{2} s} = \nrm{}{\bpa{H - \frac{L + \mu}{2} I} s} 
	%\leq \nrm{}{H - \frac{L + \mu}{2} I} \nrm{}{s} 
	\leq \frac{L-\mu}{2} \nrm{}{s} .
\]

\medskip\noindent
\textbf{Part II.} We prove the converse by construction. To this end, we make the following \emph{claim}. If  $u, v \in \mb{R}^d$, are such that $\nrm{}{u} = \nrm{}{v} = 1$, then there exists a symmetric real matrix $Q$ such that $Q u = v$ and $\lambda(Q) \subseteq \bst{-1, 1}$. To prove this, we first note that if $u = - v$ then we can choose $Q = - I$. Otherwise, let
\eqals{
	e = \frac{u + v}{\nrm{}{u + v}} .
}
Then, a simple calculation shows that 
\eqal{
	Q = 2 e e^T - I
}
satisfies $Q u = v$ and $Q^T = Q$. Since $\lambda(2 e e^T) = \bst{0, 2}$, we have $\lambda(Q) = \bst{-1,1}$, showing that our claim is true. 

%For the ``if'' part, we prove by construction. We will use the following lemma:
%\begin{framed}
%\begin{lem}
%\label{lem_existence_of_matrix_connecting_two_unit_vector}
%Let $u, v \in \mb{R}^d$, with $\nrm{}{u} = \nrm{}{v} = 1$. Then there exists a symmetric real matrix $Q$ such that $Q u = v$ and $\lambda(Q) \subseteq \bst{-1, 1}$. 
%\end{lem}
%\begin{proof}[Proof of Lemma \ref{lem_existence_of_matrix_connecting_two_unit_vector}]
%If $u = - v$ then we simply take $Q = - I$, so we can assume that $u \neq - v$. 
%
%Let
%\eqals{
%	e = \frac{u + v}{\nrm{}{u + v}}
%}
%Then, one can easily verify that 
%\eqal{
%	Q = 2 e e^T - I
%}
%satisfies $Q u = v$ and $Q^T = Q$. For the eigenvalues, note that $\lambda(2 e e^T) = \bst{0, 2}$, hence $\lambda(Q) = \bst{-1,1}$. 
%\end{proof}
%\end{framed}
Now, to prove Part II, we assume that \eqref{exclaim} holds. If 
\eqals{
	y - \frac{L+\mu}{2}s = 0 ,
}
then it follows immediately that  $y=Hs$ with  $\lambda(H) \subseteq [\mu, L]$.
Otherwise, 
%\eqals{
%	y - \frac{L+\mu}{2}s \neq 0
%}
define
\eqals{
	v = \frac{{y - \frac{L+\mu}{2}s}}{\nrm{}{y - \frac{L+\mu}{2}s}}
\quad\mbox{and} \quad
	u = \frac{s}{\nrm{}{s}} .
}
We have shown above that since $v, u$ are unit vectors, there exists a symmetric real matrix $Q \in \mb{S}_{d\times d}$ such that $v = Q u$ and $\lambda(Q) \subseteq \bst{-1,1}$, i.e.,
\eqals{
	Q \frac{s}{\nrm{}{s}} = \frac{{y - \frac{L+\mu}{2}s}}{\nrm{}{y - \frac{L+\mu}{2}s}} .
}
Hence, we have
\eqals{
	y = H s ,
}
where
\eqals{
	H = {\frac{L+\mu}{2} I + \frac{\nrm{}{y - \frac{L+\mu}{2}s}}{\nrm{}{s}}Q} .
}
Since we assume that 
\eqals{
	\frac{\nrm{}{y - \frac{L+\mu}{2}s}}{\nrm{}{s}} \leq \frac{L-\mu}{2} ,
}
and $\lambda(Q) \subseteq \bst{-1, 1}$, we conclude that
\[
	\lambda(H) \subseteq [\mu, L].
\]
\end{proof}

\newpage

\bibliographystyle{siamplain}
\bibliography{references}

\end{document}